\documentclass{paper}
\usepackage[T1]{fontenc}
\usepackage[latin1]{inputenc}
\usepackage{ngerman}
\usepackage{amssymb}
\usepackage{amsmath}
\usepackage{amsthm, marvosym} 
\usepackage{graphicx}
\usepackage{algorithm} 
\usepackage{algpseudocode} 
\usepackage{appendix}
\usepackage{hyperref}
\newcommand{\pstricks}{./}
\newcommand{\tooManySperners}{.}
\newcommand{\maxCi}{.}

\newtheorem{theorem}{Theorem}[section]
\newtheorem{definition}{Def.}[section]
\newtheorem{remark}{Remark}[section]

\newcommand\conv[1]{\operatorname{conv}(#1)}
\newcommand\realnr{\mathbb{R}}
\newcommand{\vek}[1]{\mathchoice{\displaystyle\boldsymbol{#1}}
{\textstyle\boldsymbol{#1}}{\scriptstyle\boldsymbol{#1}}
{\scriptscriptstyle\boldsymbol{#1}}}
\def\norm#1{  \left\| #1 \right\|}

\DeclareMathOperator*{\argmax}{arg\,max}
\newcommand{\simp}{\mathcal S}

\begin{document}

\selectlanguage{\english}


\title{Global Solver based on the Sperner-Lemma and Mazurkewicz-Knaster-Kuratowski-Lemma based proof of the Brouwer Fixed-Point theorem}

\author{Thilo Moshagen 
$^{\ast}$\thanks{$^\ast$Corresponding author. Email: thilo.moshagen@hs-wismar.de
\vspace{6pt}} \\ 
{\em{Hochschule Wismar, Germany}}}

\maketitle
\abstract{In this paper a fixed-point solver for mappings from a Simplex into itself that is gradient-free, global and requires $d$ function evaluations for  halving the error is presented, where $d$ is the dimension. It is based on topological arguments and uses the constructive proof of the Mazurkewicz-Knaster-Kuratowski lemma when used as part of the proof for Brouwer's Fixed-Point theorem.  }

\paragraph{Keywords:}
topological solver, bisection solver, root finding, Brouwer's Fixed Point Theorem, Mapping Degree, Sperner Lemma, Knaster-Kuratowski-Mazurkewicz-Lemma

\section{Overview}
Let
\begin{align}\label{eq:injection}
\vek F:\quad\realnr^d\supset \simp :=\text{conv}(\vek v_0,...\vek v_d)&\longrightarrow  \simp=\text{conv}(\vek v_0,...\vek v_d)
\end{align}
be a continuous mapping from $\text{conv}(\vek v_1,...\vek v_d)$ into itself. Brouwer's Fixed Point theorem for simplices states that   there is a Fixed Point 
\begin{equation}\vek F(\vek x)=\vek x.
\end{equation} 
The popular proof of it (see section \ref{sec:brouwerFP}) using the  Knaster-Kuratowski-Mazurkewicz lemma (\emph{KKM-lemma}, Theorem \ref{th:kkm}) uses points' change of distance to simplex corners under the mapping - points whose distance to all corners is not reduced are fixed points. The KKM-lemma's proof contains an algorithm to find those points.\\ 
In the Appendix \ref{sec:injectivity} it is shown that most zero search problems can be transformed into the above Fixed Point shape and are in the scope of this contribution.

\subsection{Outline of Working Principles }

\begin{figure}[h!tb]
\begin{center}
\includegraphics[width=0.4\textwidth]{\pstricks/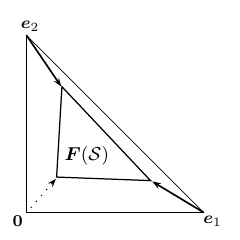}
\includegraphics[width=0.4\textwidth]{\pstricks/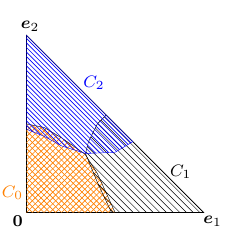}
\caption{\label{fig:cigeneral} Right: Sets $C_i$ of points that are mapped not closer to corner $\vek v_i$ by some mapping (left). The intersection of all $C_i$ are fixed points. }
\end{center}
\end{figure}
The working principle of the algorithm suggested  in this paper as well as the proof of Brouwer's Fixed-Point theorem used for it can be visualized as follows: Let the $\lambda_i(\vek u)$ denote the components of the barycentric coordinates of a point $\vek u$ with respect to the corners $\{ \vek v_j\}_{j=0,...,d}$. The bigger $\lambda_i(\vek u)$, the closer  $\vek u$ is to corner  $\vek v_i$. Thus, the sets 
\begin{equation}\label{eq:C_i}
C_i=\left\{\vek u \in S : \lambda_i(\vek F(\vek u))\le\lambda_i(\vek u)\right\}
\end{equation}
are the sets of points that are mapped away or equally far from corner $\vek v_i$, equivalently, not closer to  $\vek v_i$. A point that is mapped not closer to any corner is a fixed point, so all points in the intersection of all $C_i$ are fixed points.\\
Assume the $C_i$ look like in Fig. \ref{fig:cigeneral}. To find the fixed point, bisect the simplex, from the emerging simplices choose one with corners in all $C_i$ which always exists as Fig. \ref{fig:cigeneral} suggests and  the Knaster-Kuratowski-Mazurkewicz-Lemma (KKM) (Th. \ref{th:kkm} shows, bisect it again and so on. 

\subsection{Use of the Knaster-Kuratowski-Mazurkewicz-Lemma}

We utilize the proof of Brouwers Fixed Point Theorem using the Sperner-Lemma ( Th. \ref{th:sperner}) and KKM Lemma (Th. \ref{th:kkm}) \cite[ pp 60]{zeidler95}. This proof is constructive: According to the KKM-Lemma, a simplex with a set of closed sets $\left\{C_i\right\}_{i=0,..,d}$ with the property that each generalized face   $\text{conv}(\vek v_{i_1},...\vek v_{i_n})$ is covered by the union of the $C_i$ with the same indices, i.e. $\vek v_1\in C_1$, $\text{conv}(\vek v_1, \vek v_3)\subset  C_1\cup C_3$ et cetera, has at least one point that is in all $C_i$. From this follows Brouwers Fixed Point Theorem by considering the 
sets as defined in Eq. \eqref{eq:C_i}
where the $\lambda_i$ denote the components of the barycentric coordinates. In words, $C_i$ are the points that are mapped further away from, precisely: not nearer  to the corner $\vek v_i$. They have a common point according to the KKM Lemma, and an image point that was not mapped nearer to any corner is a fixed point.\\
Thus,  if each corner of a simplex of a triangulation $\simp_j$ is  in a different $C_i$,  a fixed point   must be close to it.
The suggested algorithm's advantages are that it is gradient-free, finds one solution globally and may find further solutions  while it might miss  solutions under certain circumstances, needs a low amount of function evaluations and is explorative, by which we mean that the data produced gives a good idea of the examined function, in particular is suitable to be used for regression.

\section{Algorithm}
\subsection{Basic Algorithm}

The suggested procedure is based on the proof of the KKM-Lemma (Theorem \ref{th:kkm}) using the Sperner-Lemma (Theorem \ref{th:sperner}). 
For any triangulation $\left\{S_k\right\}$ of $S$, the first triangulation possibly being $\simp$ itself, the nodes of the triangulation are indexed with the $i$ from the $C_i$ that contains them, possibly more than one. If a simplex has a corner in each of the $C_i$, so its corners can be numbered from 0 to $d$ by for each corner picking one of its indices, 
the simplex is called a \emph{Sperner} simplex (Definition \ref{th:defSperner}) and contains or is close to the fixed point. \\
Now that simplex is refined by adding a new point. It again is indexed by the sets $C_i$ it lies in. Thus an odd number (at least one) of the simplices of the refined triangulation again is Sperner by the Sperner Lemma, and as it is the Sperner simplex that has been divided, at least one of the newly emerging simplices is Sperner. So for a divided Sperner Simplex an odd number of - at least one -  new Sperner simplices emerge. 
The procedure is repeated and provides a sequence of simplices 
whose corners are mapped less far away from the corners of $S$ than their preimage:

\begin{algorithm}
	\caption{Basic scheme of the Knaster Solver} 
	\begin{algorithmic}[1]
		\State Given: Function $\vek F$, Initial Simplex $\simp$. Eventually make initial globally refined triangulation $\{\simp_j\}$.
		\State evaluate $\vek F$ on all vertices $\vek v_i$
		\State evaluate $C_i$ membership on all vertices $\vek v_i$
		\For {$bisectionSteps=1,2,\ldots$}
			\For {Subsimplices $\simp_j\in \{\simp_i\} $}
				\If { $\simp_j$ is Sperner }
				\State  bisect $\simp_j$ and neighbor
				\State on new vertex evaluate $\vek F$ 
				\State  on new vertex evaluate $C_i$ membership 
				\EndIf
			\EndFor
		\EndFor
	\end{algorithmic} 
\end{algorithm}
\subsection{Details of the Algorithm}\label{sec:detailsOfAlgo}
A \emph{consistent} refinement algorithm  for arbitrary dimensions is needed, which means that a node inserted on an edge of a simplex should become corner of all simplices emerging from the refinement. That is, there should be no hanging nodes. Furthermore, it should not produce small angles between edges as this would mean that already short edges are refined before longer ones. Preferably,  it should reduce the maximum edge length in each step. \\
\begin{algorithm}
	\caption{\label{algo:knaster} Knaster Solver. For explanations see \ref{sec:detailsOfAlgo}} 
	\begin{algorithmic}[1]
		\State Given: Function $\vek F$, Initial Simplex $\simp$. Eventually make initial triangulation $\{\simp_j\}$.
		\State vectors of edges, edge ages, memberships of simplices for each node.
		\State Initialize edge age vector, edges $(\vek 0, \vek e_i)$ with 0 and $(\vek e_i, \vek e_j)$ with 1
		\State evaluate $\vek F$ on all vertices $\vek v_i$
		\State evaluate $C_i$ membership on all nodes $\vek v_i$ and Sperner Property on all $ \simp_j $
		\For {$bisectionSteps=1,2,\ldots$}
			\State Increase EdgeAges by 2
			\For {Edges $(\vek u_i, \vek u_j)$}
			 	\If {edgeAge>3+ageCount}
				\If {Any of Subsimplices $\simp_k$ with $(\vek u_i, \vek u_j)\in \simp_k$ is Sperner}
					\State create $\vek u_\text{new}$ in the middle of oldest edge, add it to node list, 
					\State on new Point evaluate $\vek F$ and $C_i$ membership 
					\State add  $(\vek u_i, \vek u_\text{new})$,  $(\vek u_\text{new}, \vek u_j)$ to edge list, set their edge age to 0
					\State  bisect all $\simp_j$ that shared the split edge
					\State add new simplices to simplex list, 
					\State add new edges to edge list, set their edge age to 1.
					\State refFlag=1
				\EndIf
				\EndIf
			\EndFor
			\If {refFlag}
			\State ageCount = ageCount+1
			\Else
				\For {Edges $(\vek u_i, \vek u_j)$}
			 		\If {edgeAge>2+ageCount}
						\State bisect around edge as above 
					\EndIf
				\EndFor
			\EndIf
		\EndFor
	\end{algorithmic}
	\label{algoII} 
\end{algorithm}
These requirements are met by an algorithm that places a point in the middle of the longest edge of a simplex marked for refinement, and then splits all simplices that share this edge (Algorithm \ref{algoII})

 The \emph{edge age} value of 0 is assigned to the bisected edge only, age one is assigned to all new edges that are shorter than those of the old Sperner simplex, but not halved. All other edge's age are incremented by two. Thus, it is guaranteed that a newly bisected edge is the last in the file of  edges that are candidates for refinement. In this form, the algorithm converges as described in section \ref{sec:convergence}.\\
 The \emph{age count} construction  establishes that the algorithm does not run idly until edges grow old enough for refinement.

\subsection{Situations that make the algorithm miss Fixed Points}
The algorithm detects $\simp_j$ with corners in all $C_j$. The fixed point is where the boundaries of all $C_i$'s  meet \eqref{eq:C_iboundary}. An even number of fixed points is missed by the algorithm if a bounded subset of $C_i$ is subset of a $\simp_j$ (Figure \ref{fig:situations}, right). In general, a fixed point might be missed in all situations where for one $C_i$ $C_i\cap \simp_j\neq \emptyset$ but $C_i$ contains no corner of $\simp_j$. In these two situations  initial refinement of the triangulation may reveal further fixed points. \\
The algorithm is guaranteed to find one fixed point and likely to find a good part of the fixed points.
A fixed point may  not be inside the Sperner simplex that it causes. Nevertheless, the sequence of Sperner simplices approaches the fixed point (Figure \ref{fig:situations}, middle).
\begin{figure}
\includegraphics[width=0.3\textwidth,trim= 4cm 3.0cm 4cm 3.0cm,clip]{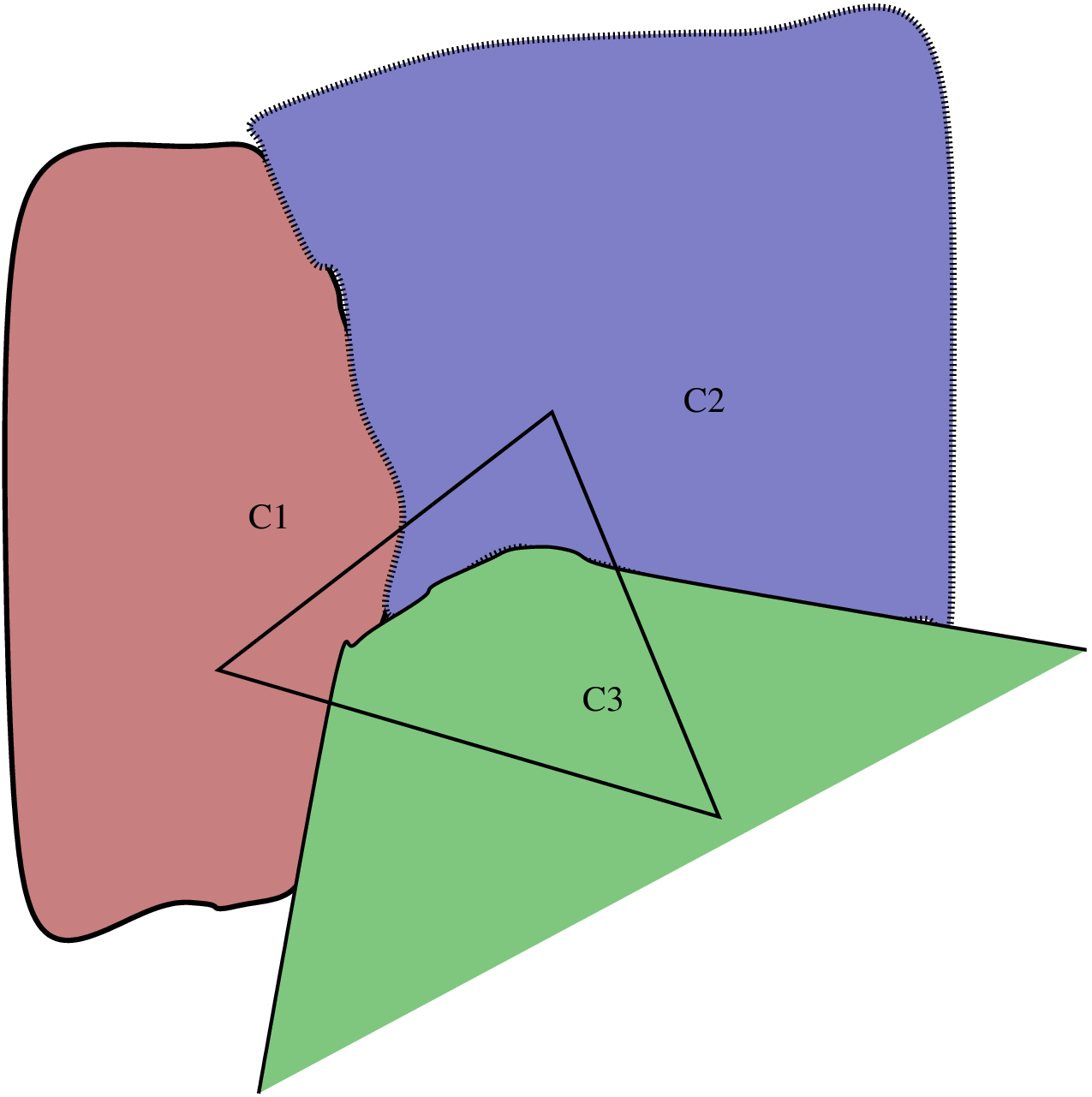}
\includegraphics[width=0.3\textwidth,trim= 4cm 3.0cm 4cm 3.0cm,clip]{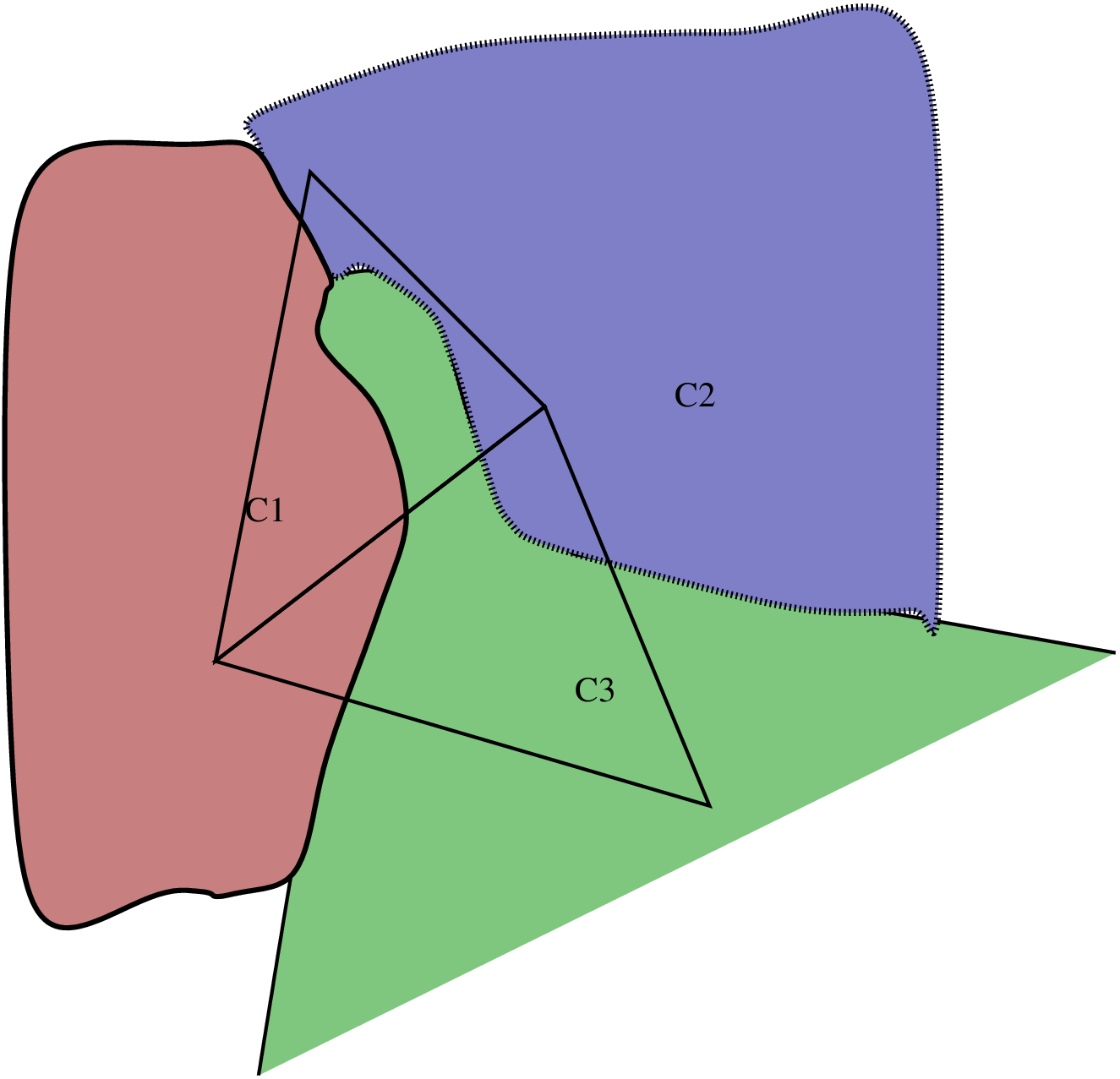}
\includegraphics[width=0.3\textwidth,trim= 3.6cm 2.1cm 3.6cm 3cm,clip]{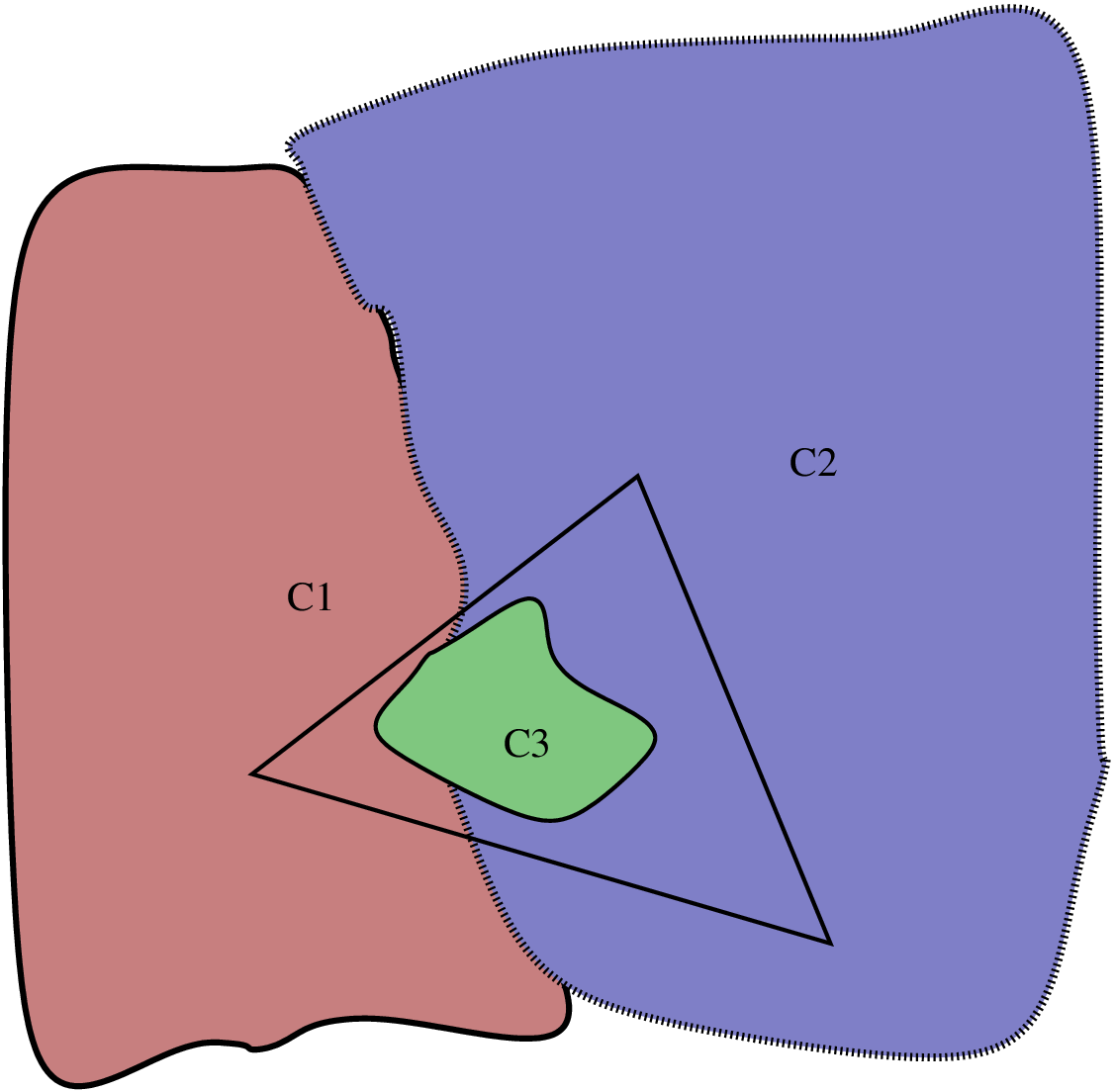}
\caption{\label{fig:situations}Left: The fixed point is inside the Sperner simplex and will be in such after refinement. Middle:  The fixed point is not inside the Sperner simplex, but in the non-Sperner above. Right: A simplex is non-Sperner and contains a pair number of fixed points.}
\end{figure}

\subsection{Convergence}\label{sec:convergence}
We denote by error the maximum distance between possible positions of the fixed point.\\

The algorithm yields in the best case as many sequences $(\simp_i)_{i=1,...}$ of Sperner simplices as there are fixed points. They obey $\simp_i\cap\simp_{i+1}\neq \emptyset$, as for each divided Sperner simplex a new one emerges, but  not necessarily $\simp_i\supset\simp_{i+1}$. Thus it cannot be generally  said that $\simp_{i+1}$ has half the volume as $\simp_i$ and a reduced maximum diameter. This makes a  general convergence and computational cost estimate difficult. But recall that convergence statements for Newton-like solvers require an
initial proximity of starting point and zero as well. Thus, a  convergence and computational effort statement is given for the situation $\simp_i\supset\simp_{i+1}$, which is shown in Fig. \ref{fig:situations}, left image:\\
For simplicity and without loss of generality (by affine mapping), we consider the simplex with $\vek v_0=\vek 0$ and $\vek v_i= \vek e_i$. Independent of the dimension, the edges containing $\vek 0$ are of  length 1 and those between $\vek v_i$ and $\vek v_j$ are of length $\sqrt{2}$. The longest edges are refined first and this behavior is inherited to the child simplices by the edge age mechanisms as long as it is always a child and not a neighbour that is the emerging Sperner simplex. 

\begin{figure}[h!tb]
\center{\includegraphics[width=0.3\textwidth]{\pstricks/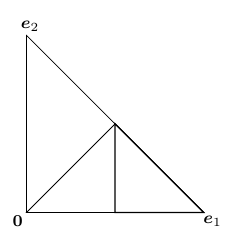}}
\caption{\label{fig:simplex2d} The $2-d$ unit simplex refined twice.}
\end{figure}
With these facts and circumstances, the two-dimensional unit simplex needs to be refined twice 
to halve all edge's lenghts, as Fig \ref{fig:simplex2d} shows. To prove this statement, 
first observe that it is independent on which subsimplex is refined again due to symmetry of 
the division. Further, the subsimplex containing $\vek e_1$ yielding from the refinement is 
congruent to the original, as one of its edges lies on the  line  $\vek 0, \vek e_1$ and 
 one on the  line  $\vek e_1, \vek e_2$, so in original direction, and those edges have half the length of the edges they  originate from, thus the two mentioned new points 
 result from moving $\vek 0$ and  $\vek e_2$ 
half the distance in direction towards  $\vek e_1$, which is congruency.

\begin{figure}[h!tb]
\center{\includegraphics[width=0.24\textwidth]{\pstricks/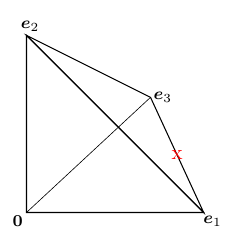}
\includegraphics[width=0.24\textwidth]{\pstricks/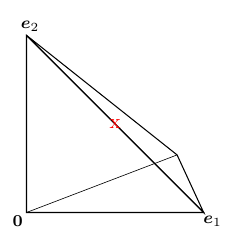}
\includegraphics[width=0.24\textwidth]{\pstricks/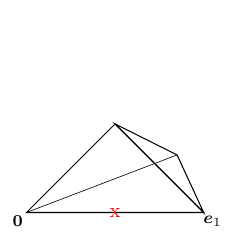}
\includegraphics[width=0.24\textwidth]{\pstricks/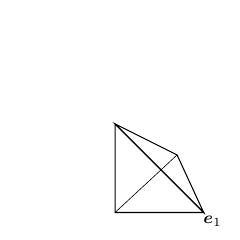}
}
\caption{\label{fig:simplex3d} The $3-d$ unit simplex refined three times.}
\end{figure}

For the 3-dimensional unit simplex $\conv{\vek 0, \vek e_1, \vek e_2, \vek e_3}$ an edge containing $\vek e_3$ is refined, without loss of generality $\conv{\vek e_1, \vek e_3}$.  The new edge has the same direction as $\conv{\vek e_1, \vek e_3}$ and half the length.\\
Then the 2d-face $\conv{\vek 0, \vek e_1, \vek e_2}$ is refined, and as said for 2-d-Simplices, after two refinements yielding a 2d-subsimplex which is congruent to $\conv{\vek 0, \vek e_1, \vek e_2}$ and has half the edge lenghts. Without loss of generality let $\vek e_1$ be corner of the emerging subsimplex. Its edges even have the same directions as those of $\conv{\vek 0, \vek e_1, \vek e_2}$. The new edge yielding from dividing $\conv{\vek e_1, \vek e_3}$ as well has old direction, so there are three edges with original direction and half length, meaning that each of their three end points can be produced by moving one $\vek e_i$, $i\neq 1$ half the way towards $\vek e_1$. 
Thus  
the refined simplex is congruent to the initial simplex, scaled by $\frac12$. Completing the proof, this holds for all possible subsimplices emerging by three refinements as specified due to the symmetry of the bisection, so by three consecutive refinements all edge lengths are halved.

We generalize this result to $d$ dimensional simplices by induction:
\begin{remark} 
The following result is unlikely to be new, yet, we could not find it anywhere.
\end{remark}
\begin{theorem}\label{th:edgeLengths}
If a unit simplex  $\conv{\vek 0, \vek e_1, ... \vek e_d}$ is refined such that a new node is inserted always on one of the longest edges, then after $d$ refinements the length of all edges is halved.
 \end{theorem}

\begin{proof}(By induction.)
The claim is true for $d=2$ an $d=3$. Assume it is true for $d-1$.\\
The $d$-dimensional simplex $\conv{\vek 0, ... \vek e_d}$ is initially refined by introducing a new vertex, without loss of generality on edge $\conv{\vek e_{d-1}, \vek e_d}$.\\
Then the $d-1$-face $\conv{\vek 0,... \vek e_{d-1}}$, we shall call it the front side, is refined, and according to the induction prerequisite, all  $d-1$-subsimplices that emerge  after $d-1$ refinement steps are congruent to the front face. We show that one of the emerging $d$-subsimplices is congruent to the original - then all are due to the symmetry of the bisection. Consider the emerging $d$-subsimplex containing $\vek e_{d-1}$. Its face in the front side $\conv{\vek 0,... \vek e_{d-1}}$ is congruent scaled by $\frac12$ to the front side by the prerequisite, and its edge in direction $\vek e_d$ is trivially in direction of its pre-refinement ancestor $\conv{\vek e_{d-1}, \vek e_d}$. Thus, the end vertices of $d$ edges are on $d$ lines identical with the lines through $\vek e_d$ 
$\vek 0 , \vek e_1, ..., \vek{ e}_{d-2}$ and $\vek e_d$, 
with $\frac12$ the original distance, meaning that the original simplex has been scaled by $\frac12$, which is the congruency of the subsimplex with the original.
\end{proof}
Applied to the sequence of Sperner simplices yielding from  the suggested algorithm, this is:
\begin{theorem}\label{th:convergence}
In a sequence of Sperner simplices from Algorithm \ref{algoII} with $\simp_i\supset \simp_{i+1}$, after $d$ steps requiring $d$ function evaluations the error is halved.\\
 \end{theorem}
 
\subsection{Two Examples}

\begin{figure}[h!tb]
\includegraphics[width=0.5\textwidth]{\pstricks/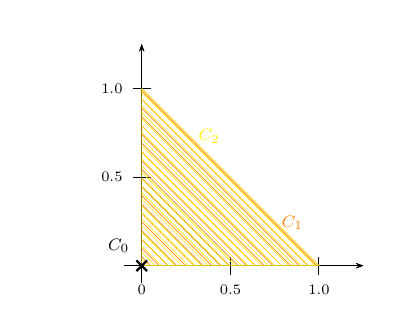}
\includegraphics[width=0.5\textwidth]{\tooManySperners/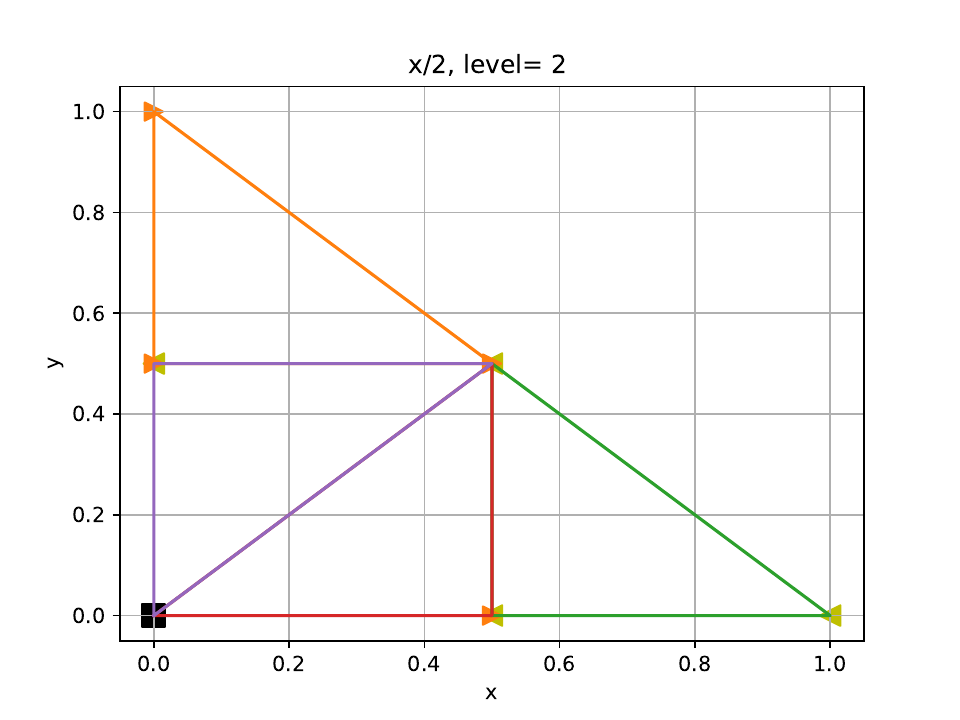}\\
\includegraphics[width=0.5\textwidth]{\tooManySperners/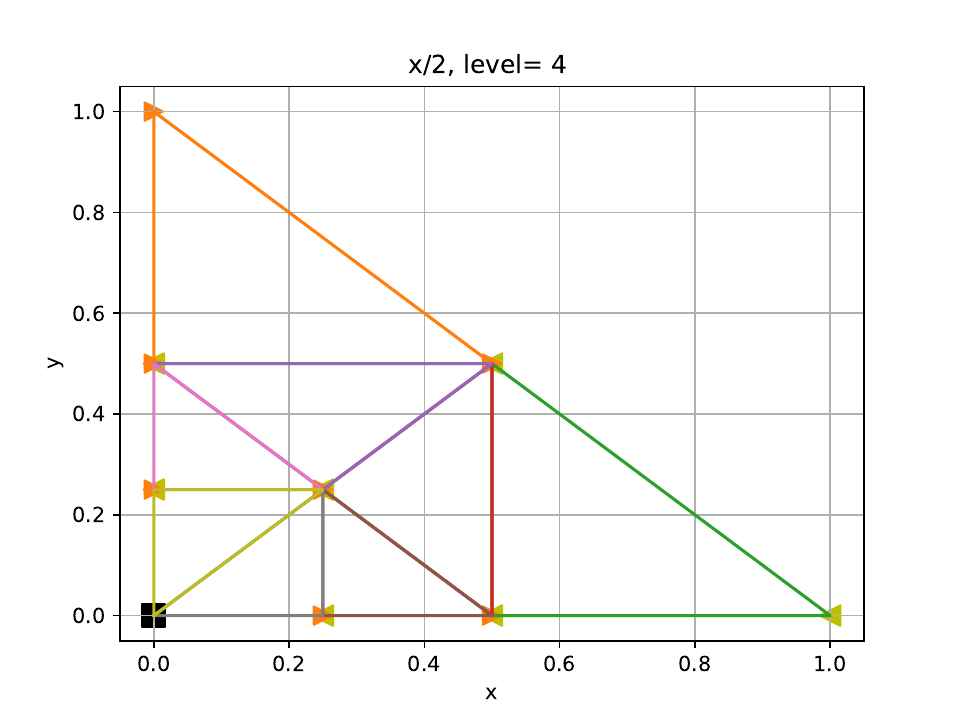}
\includegraphics[width=0.5\textwidth]{\tooManySperners/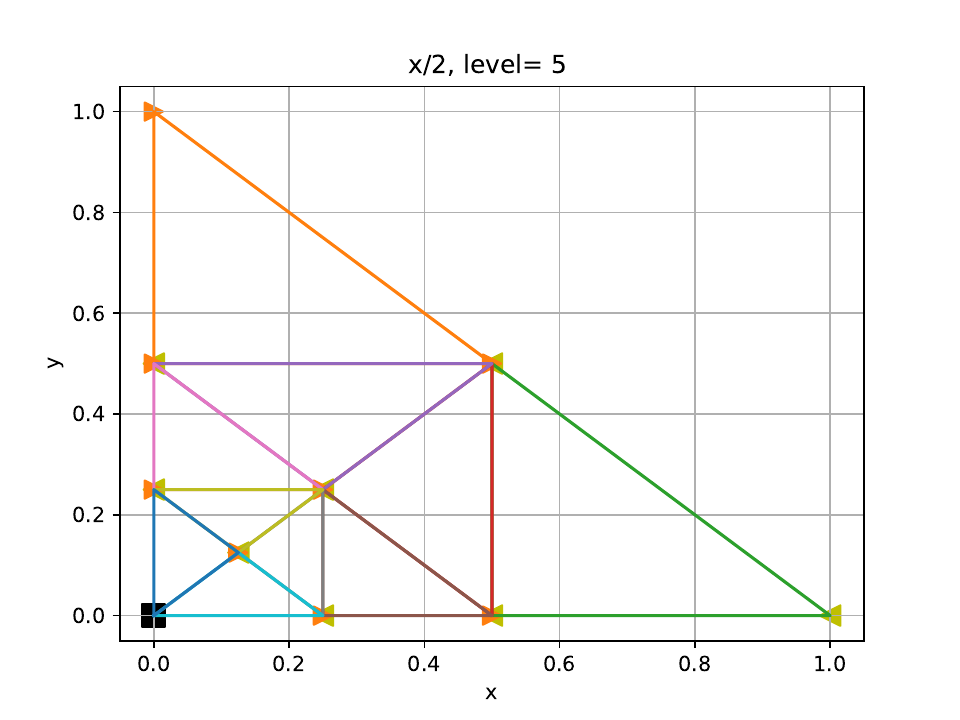}
\caption{\label{fig:xHalf2d}Above left: The $C_i$ of example in Eq.\eqref{eq:x/2}: $(x,y)^\intercal  \mapsto \frac12 (x,y)^\intercal$. All $\vek x$ are mapped farther away from $(1,0)^\intercal$, so in $C_1$ (orange), same with $C_2$ (yellow). Only the origin is mapped not closer to $(0,0)^\intercal$, so in $C_0$ (black), and is the fixed point.   Right and below: Knaster algorithms refinement steps 2, 4 and 5 for example \eqref{eq:x/2}. The coloured markers show the $C_i$ memberships: All points are in $C_1$ and $C_2$, so orange and yellow, only the origin is in $C_0$ as well and black. }
\end{figure}

The mapping 
\begin{align}\label{eq:x/2}
\vek F:\quad\realnr^2\supset \simp:=\text{conv}(\vek e_1,\vek e_2) &\longrightarrow  \simp\\
(x,y)^\intercal & \mapsto \frac12 (x,y)^\intercal
\end{align}
has the fixed-point $(0,0)^\intercal$. The predicted convergence results from Th. \ref{th:convergence} is fully met (Fig. \ref{fig:xHalf2d}).\\
The sets $C_i$, the sets of points that are mapped not farther away from corner $i$ under the mapping, are given in the same figure. Indeed, fixed-point $(0,0)^\intercal$ would be found at once if tested for its $C_i$ memberships, anyway this is an useful example for its simplicity.\\
\begin{figure}[h!tb]
\includegraphics[width=0.5\textwidth]{\pstricks/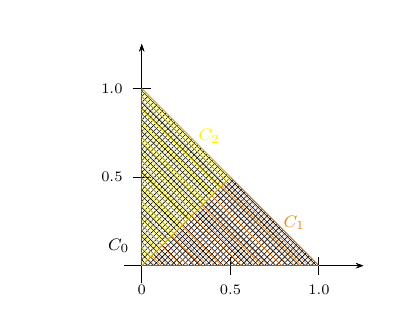}
\includegraphics[width=0.5\textwidth]{\tooManySperners/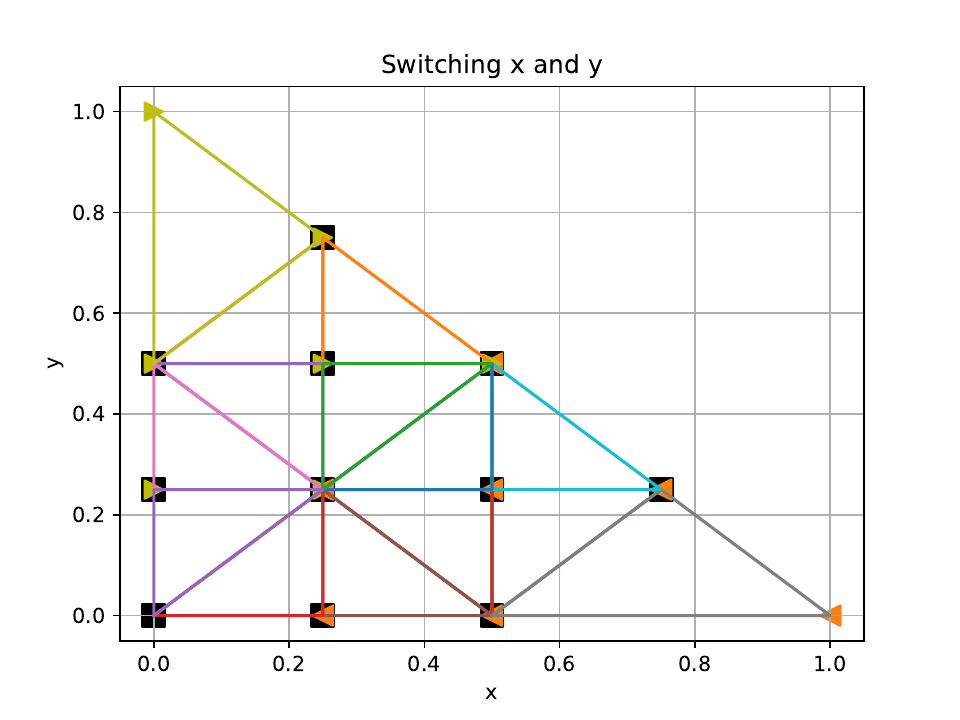}
\includegraphics[width=0.5\textwidth]{\tooManySperners/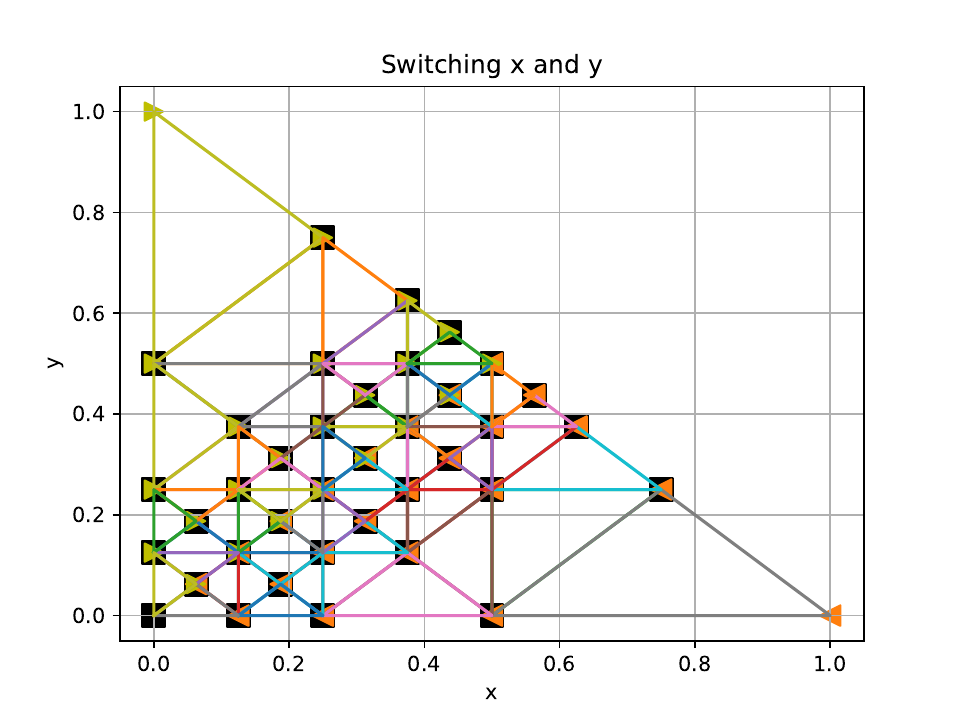}
\includegraphics[width=0.5\textwidth]{\tooManySperners/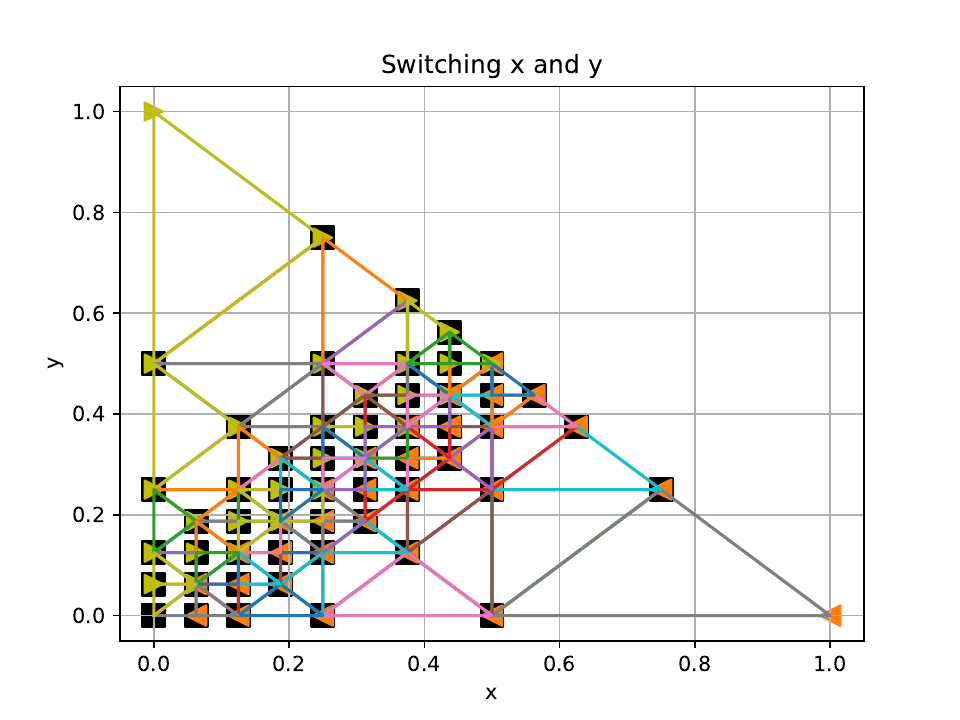}
\caption{\label{fig:switchxy}Above left: The $C_i$ of example in Eq. \eqref{eq:switchxy}: $(x,y)^\intercal  \mapsto  (y,x)^\intercal$ (switching $x$ and $y$). No point moves away from the origin, so all $\vek x$ are in $C_0$. Subdiagonal points move away from $(1,0)^\intercal$, those above the diagonal away from $(0,1)^\intercal$.   Right and below: refinement produced by the algorithm applied to example \eqref{eq:switchxy}. }
\end{figure}
The mapping 
\begin{align}\label{eq:switchxy}
\vek F:\quad\realnr^2\supset \simp:=\text{conv}(\vek e_1,\vek e_2) &\longrightarrow  \simp\\
(x,y)^\intercal & \mapsto  (y,x)^\intercal
\end{align}
has  $C_0=\simp$, $C_1$ is the area above the diagonal, and $C_2$ is the area below the diagonal, thus the fixed-points are the common set of them, the diagonal $\{(x,y)^\intercal:\, y=x\}$. It is harder to see that the convergence result is met, but by the refinements around the diagonal it is obvious that the algorithm works efficiently (Fig. \ref{fig:switchxy}). 

\section{Exploding number of Sperner Simplices}

The convergence result Th. \ref{th:convergence} assumes that a Sperner simplex upon bisection produces one new Sperner simplex, either a child of the Sperner simplex or child of a neigbours a neighbour that was divided because it shared the divided edge. The following example illustrates that more than one Sperner simplex emerges during division once nodes are in more than one $C_i$, which happens  if the $C_i$ are not disjoint. This lowers the efficiency of the algorithm.\\

\subsection{Two Examples for a rising number of Sperner Simplices}

\begin{figure}[h!tb]
\includegraphics[width=0.5\textwidth]{\pstricks/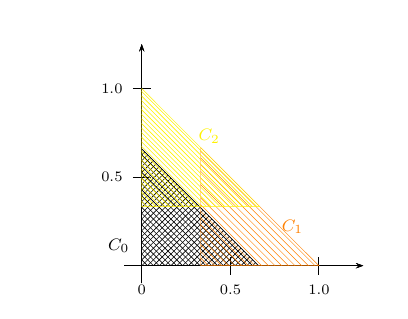}
\includegraphics[width=0.5\textwidth]{\tooManySperners/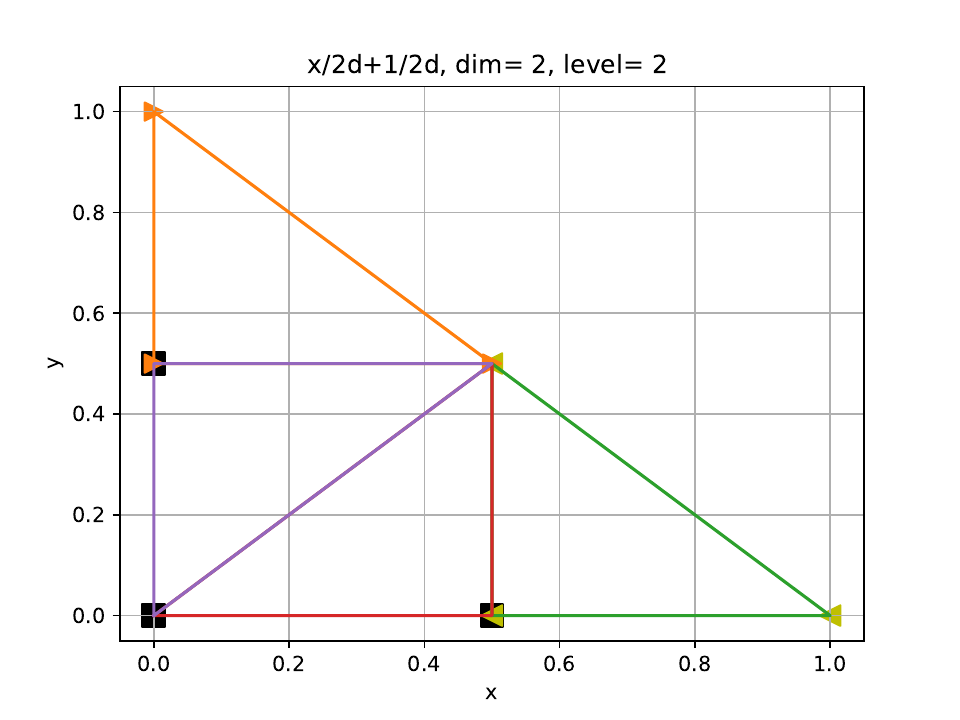}\\
\includegraphics[width=0.5\textwidth]{\tooManySperners/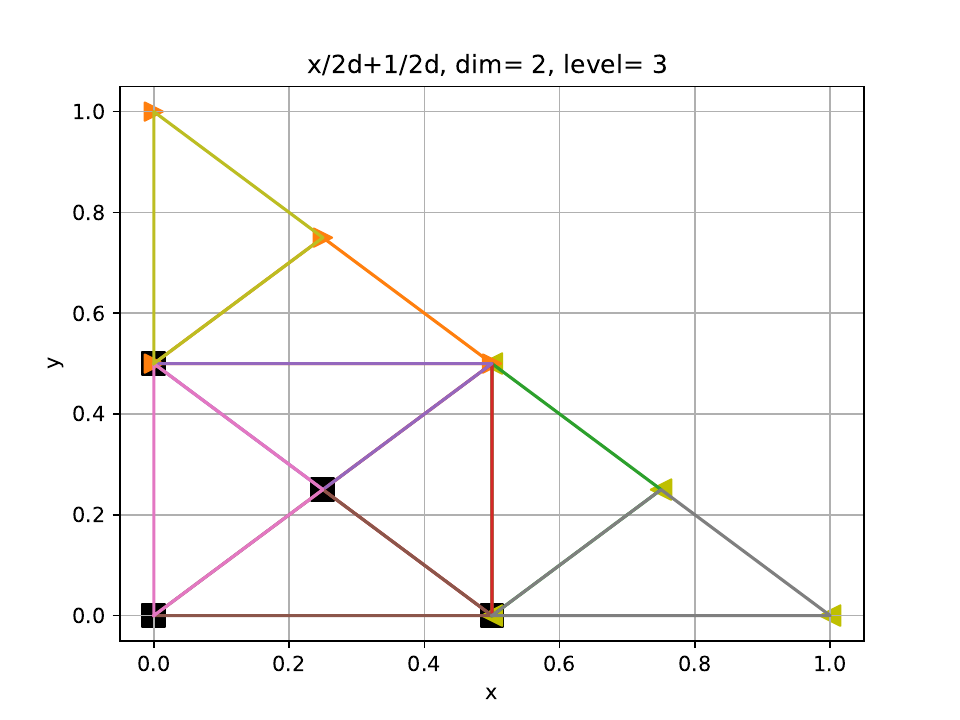}
\includegraphics[width=0.5\textwidth]{\tooManySperners/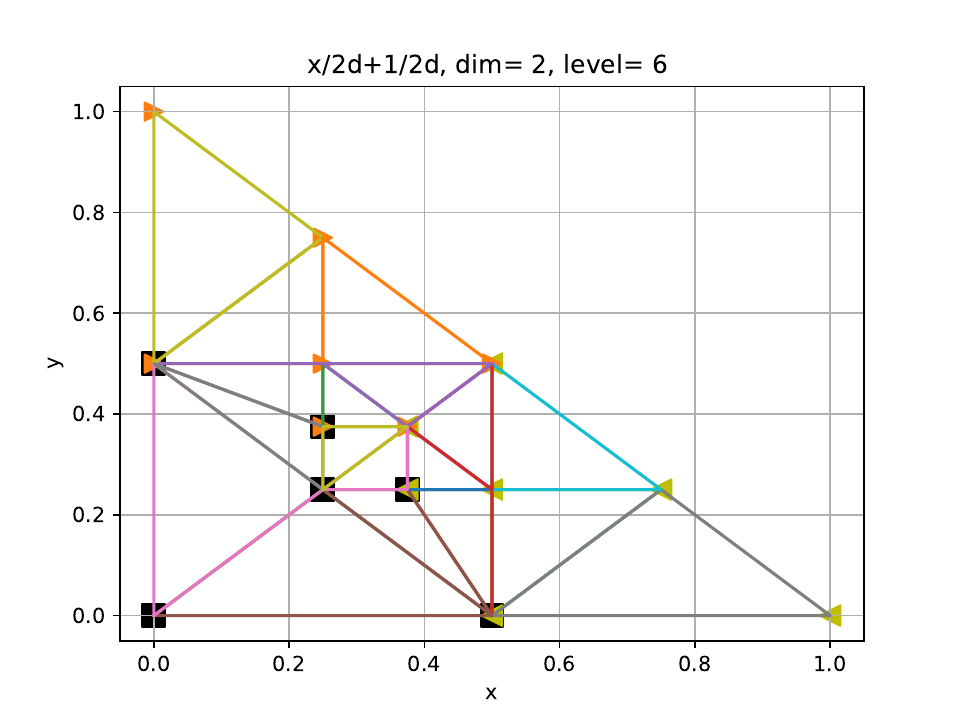}
\caption{\label{fig:x/4+1/4}Above left: The $C_i$ of example in Eq. \eqref{eq:x/4+1/4}: $(x,y)^\intercal  \mapsto \frac{1}{4}\vek x+\frac{1}{4}\vek 1$.   Above right: Step 2 of the algorithm applied to working example \eqref{eq:x/2}.  below left: Step 3. The triangles with $x\le \frac12$ and $y\ge\frac12$ should not have been refined, and it would be sufficient to work only on one side of the diagonal. Below right: Refinement to step  6. The plot shows further unnecessary  refinements. }
\end{figure}

\begin{figure}[h!tb]
\includegraphics[width=0.5\textwidth]{\tooManySperners/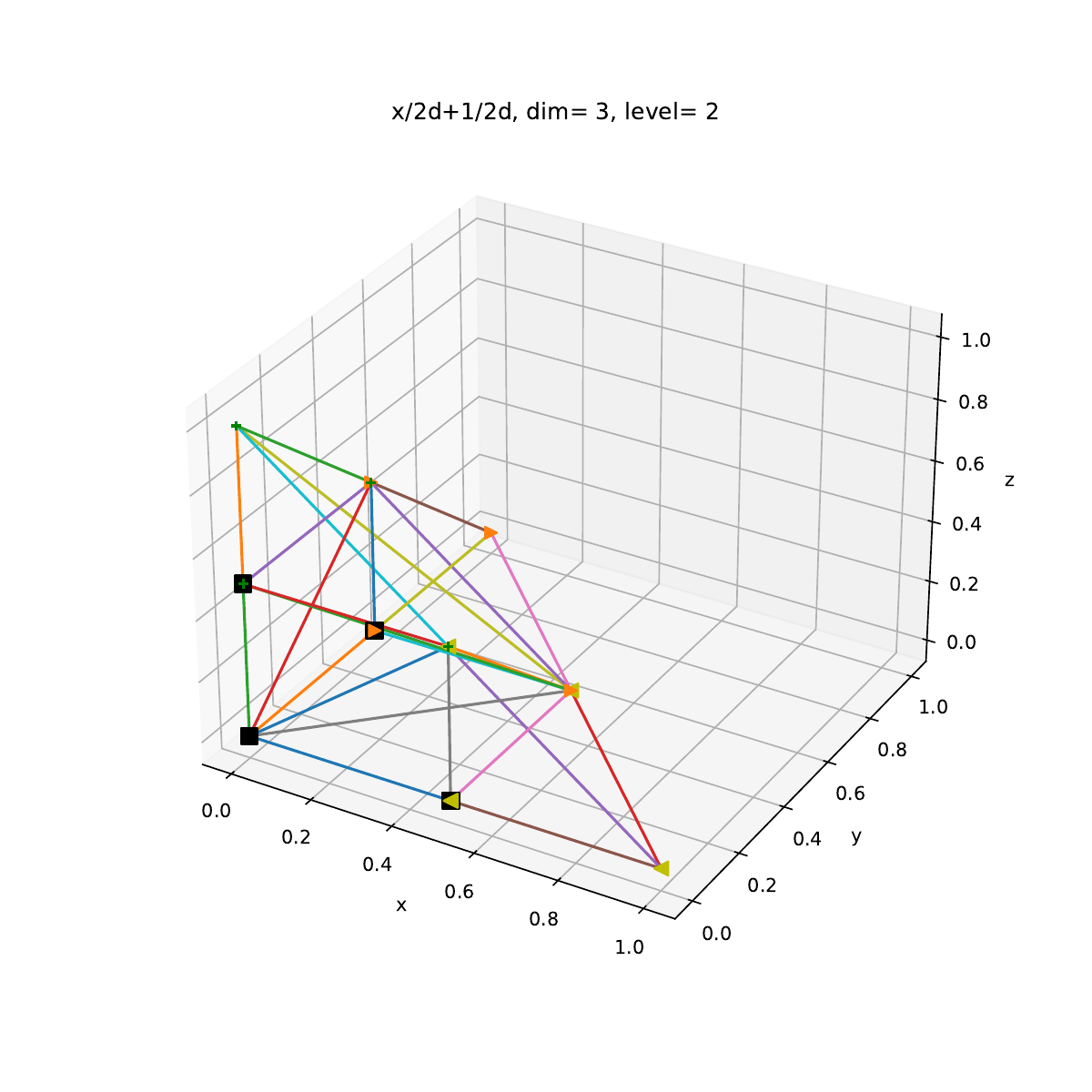}
\includegraphics[width=0.5\textwidth]{\tooManySperners/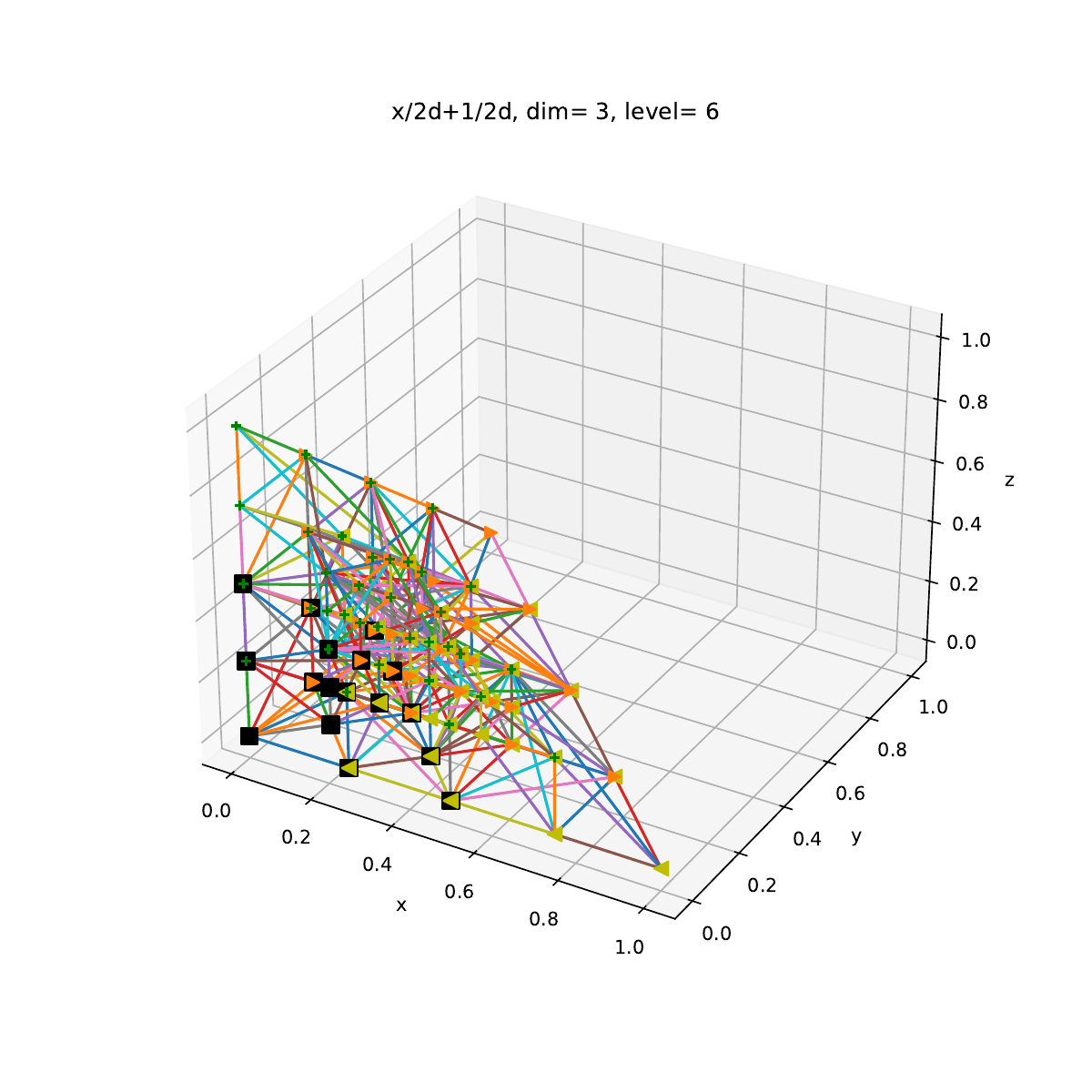}
\caption{\label{fig:x/6+1/6}Left: Problem in Eq.\eqref{eq:x/4+1/4} for $d=3$.   Right: refinement around the fixed point is visible, but similarly to the $2d$ case, the problem  exhibits inefficient divisions of simplices.}
\end{figure}

\subsubsection{Example with one fixed point}
\begin{table}[h!tb]
\begin{minipage}{0.4\textwidth}
\begin{tabular}{|c c c|}
\hline Nr. & x & y \\ \hline
11&0.50000&0.25000\\
12&0.25000&0.50000\\
...&0.37500&0.37500\\
&0.37500&0.25000\\
&0.25000&0.37500\\
&0.43750&0.31250\\
&0.31250&0.43750\\
&0.31250&0.31250\\
&0.37500&0.31250\\
&0.31250&0.37500\\
&0.34375&0.34375\\
 \hline
\end{tabular}\\
\end{minipage}
\begin{minipage}{0.5\textwidth}

\begin{tabular}{|c c c c| c|}
\hline Nr& x & y & z & e \\ \hline
95&	0.18750&0.25000&0.25000 &0.0718\\
...&	0.25000&0.18750&0.12500 &0.0910\\
&	0.18750&0.25000&0.12500 &0.0910\\
&	0.21875&0.21875&0.18750 &0.0293\\
&	0.31250&0.15625&0.21875 &0.1221\\
&	0.15625&0.31250&0.21875 &0.1221\\
&	0.18750&0.21875&0.21875 &0.0293\\
&	0.28125&0.18750&0.15625 &0.0931\\
&	0.34375&0.21875&0.12500 &0.1632\\
&	0.18750&0.28125&0.15625 &0.0931\\
104&	0.21875&0.34375&0.12500 &0.1632 \\ 
\hline
\end{tabular}
\end{minipage}
\caption{\label{tab:ptsx2d+1_2d} The last points found  for problems in Eq. \eqref{eq:x/4+1/4} for 2d, ($\vek u_{FP} = \left(\frac13,\frac13\right)^\intercal$),  (left), and 10 points of the 3d solution path ($\vek u_{FP} = \left(\frac15,\frac15,  \frac15\right)^\intercal$) (right), the error $e:= \norm{\vek x-\vek x_{FP}}$ is given for easier reading. It is the high number of Sperner simplices that makes convergence invisible for the 3d case.}
\end{table}

The mapping
\begin{equation}\label{eq:x/4+1/4}
\vek y(\vek x)= \frac{\vek x}{2d}+\frac{1}{2d}\vek 1
\end{equation}
has one fixed point at $\vek x_F = \frac{1}{2d-1}\vek 1$, so for $d=2:$ $\vek x_F = \left(\frac13,\frac13\right)^\intercal$ and for $d=3$ $\vek u_{FP} = \left(\frac15,\frac15,  \frac15\right)^\intercal$.
Let $d=2$. For this mapping, $C_1=\{\vek x: x_1> \frac13  \}$, $C_2=\{\vek x: x_2> \frac13   \}$ and $C_0=\{\vek x: x_2\le -x_1+\frac23 \}$\footnote{Note that here $\lambda_1=x_1$ and $\lambda_2=x_2$. Then 
\begin{align*}\lambda_0=1-\lambda_1-\lambda_2=1-x_1-x_2\ge &1 - (\frac14 x_1+\frac14)-(\frac14 x_2+\frac14)\\
\Leftrightarrow -3x_1-3x_2 \ge 2\\
\Leftrightarrow  x_2\le -x_1+\frac23.
\end{align*} } (Fig. \ref{fig:x/4+1/4}, top left).
Thus the diagonal $x_1=x_2$ is in $C_1$ and $C_2$, the node $(\frac12, 0)$ is in $C_0$ and $C_2$ and the node $(0,\frac12)$ is in $C_0$ and $C_1$. Those nodes produce Sperner simplices, which are marked for refinement and produce computational effort (Fig. \ref{fig:x/4+1/4}, top right and bottom left).\\
 As the number of simplices that a node is a corner of grows with the dimension,  this effect becomes stronger with growing dimension, as the example   \eqref{eq:x/4+1/4} with $d=3$ shows (Fig. \ref{fig:x/6+1/6}). The convergence becomes barely visible due to the high number of emerging sperner simplices (see Table \ref{tab:ptsx2d+1_2d}).

\section{Overcoming Sperner Simplex Number Explosion by unifying $C_i$ Membership}

\subsection{Redefining the $C_i$}
There are ways to reduce the number of $C_i$ memberships of nodes such that the prerequisite of the 
KKM lemma \ref{th:kkm}
\begin{equation*}
\vek u_i \in \text{conv}(\vek v_{i_1},...\vek v_{i_n})\,\Rightarrow \,i\in \{i_1,...,i_n\} \end{equation*}
is still met. To show that the emerging new $C_i$ may be such that fixed points are lost, a brute force algorithm for discarding $C_i$ memberships is presented in the appendix
\ref{sec:reduceC_iBruteForce}.

The following algorithm  does not conceal  fixed points. Consider the following way to define $C_i$: 
 \begin{equation}\label{eq:C_iAsMax}
\vek u \in C_i \quad \Leftrightarrow \quad i= \argmax_j \lambda_j(\vek u) - \lambda_j(\vek F(\vek u)).
\end{equation}
As $\argmax$ gives more than one number if the maximum is attained for more than one $j$, the $C_i$ are again not disjoint, but the intersections of the $C_i$ should in most situations be the boundaries of the $C_i$ and of measure 0. This reduces the number of emerging Sperner simpices vastly. It remains to investigate in which situations intersections are not of measure 0. This is left to future work. 
We 
\begin{enumerate}
\item 
state that in contrast to Sec. \ref{sec:reduceC_iBruteForce},  all  fixed points remain in the intersection of the $C_i$: $\vek u_{FP}\in C_i$, $i=0,...,d$, as for fixed points $\lambda_j(\vek u) - \lambda_j(\vek F(\vek u))=0$ for all $j$, and thus potentially can be found by the suggested algorithm. Thus no fixed point is lost to the algorithm.
\item show that the emerging indexing of nodes by the index of the $C_i$ membership still has the property \eqref{eq:indexedAsFace} from the Sperner lemma, 
\begin{equation*}
\vek u_i \in \text{conv}(\vek v_{i_1},...\vek v_{i_n})\,\Rightarrow \,i\in \{i_1,...,i_n\} \end{equation*}
\end{enumerate}
Let $\vek u \in \text{conv}(\vek v_{i_1},...\vek v_{i_n})$. Let $\vek u$ be a fixed point. Then   $\vek u \in C_i$ for all $i$, including $i_1, ..., i_n$.\\
Now let $\vek u\in \text{conv}(\vek v_{i_1},...\vek v_{i_n})$ be no fixed point. Assume that $\vek u \notin C_i$ for all $i=i_1,...,i_n$.
Then 
 \begin{equation}
\max_{j=1,...d}( \lambda_j(\vek u) - \lambda_j(\vek F(\vek u)))=\max_{j\notin\{i_1,...i_n\}}( \lambda_j(\vek u) - \lambda_j(\vek F(\vek u))).
\end{equation}
Furthermore, as $\vek x$ is not a fixed point, there is a $j$ such that $\lambda_j(\vek u) > \lambda_j(\vek F(\vek u))$, and thus 
 \begin{equation}
\max_{j}( \lambda_j(\vek u) - \lambda_j(\vek F(\vek u)))>0.
\end{equation}
Altogether
 \begin{equation}
 \begin{split}
0 < \max_{j=1,...d}( \lambda_j(\vek u) - \lambda_j(\vek F(\vek u)))=\max_{j\notin\{i_1,...i_n\}}( \lambda_j(\vek u) - \lambda_j(\vek F(\vek u)))\\
=\max_{j\notin\{i_1,...i_n\}}( 0 - \lambda_j(\vek F(\vek u)))\le 0,
\end{split}
\end{equation}
which is a contradiction. Thus  $\vek u \in C_i$ for an $i\in\{i_1,...i_n\}$.\\
As said, the definition of the $C_i$ is also leaving all fixed points in the intersection of the $C_i$: $\vek x_{FP}\in C_i$, $i=0,...,d$.

\begin{figure}[h!tb]
\includegraphics[width=0.5\textwidth]{\pstricks/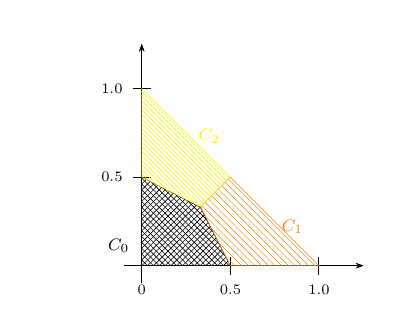}
\includegraphics[width=0.5\textwidth]{\maxCi/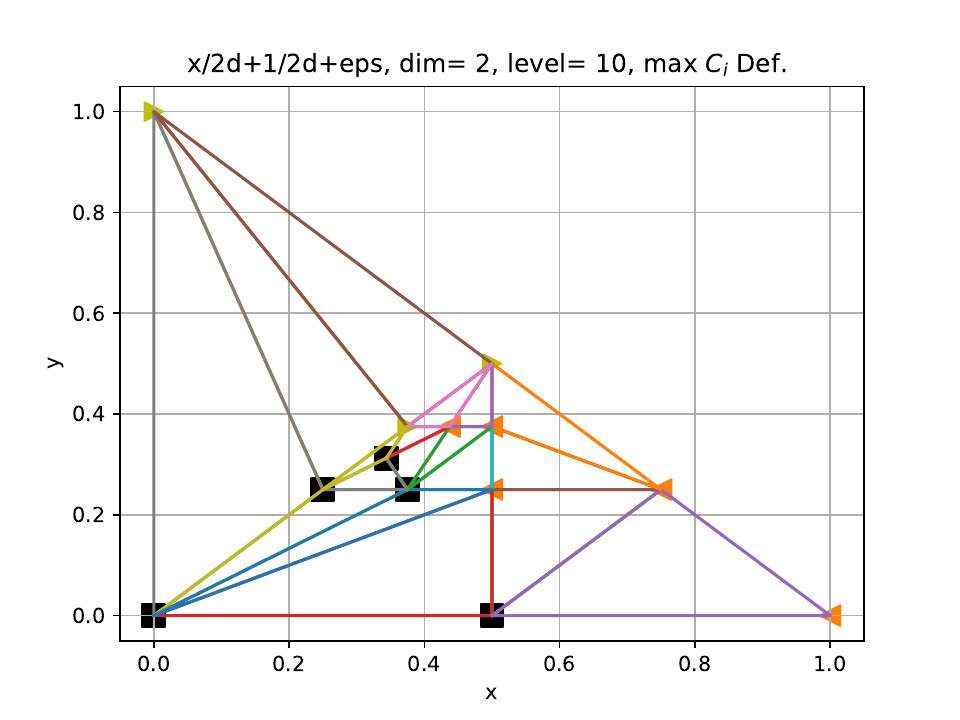}
\caption{\label{fig:x/4+1/4NewCi} Left: The $C_i$ of example in Eq. \eqref{eq:x/4+1/4+eps}: $(x,y)^\intercal  \mapsto  \frac{1}{4}\vek x+\frac{1}{4}\vek 1+\vek \epsilon$, using $C_i$ definition \eqref{eq:C_iAsMax}.   Right: Refinement pattern for that example. See Fig. \ref{fig:x/4+1/4} for comparison.}
\end{figure}

\begin{figure}[h!tb]
\begin{center}
\includegraphics[width=0.6\textwidth]{\maxCi/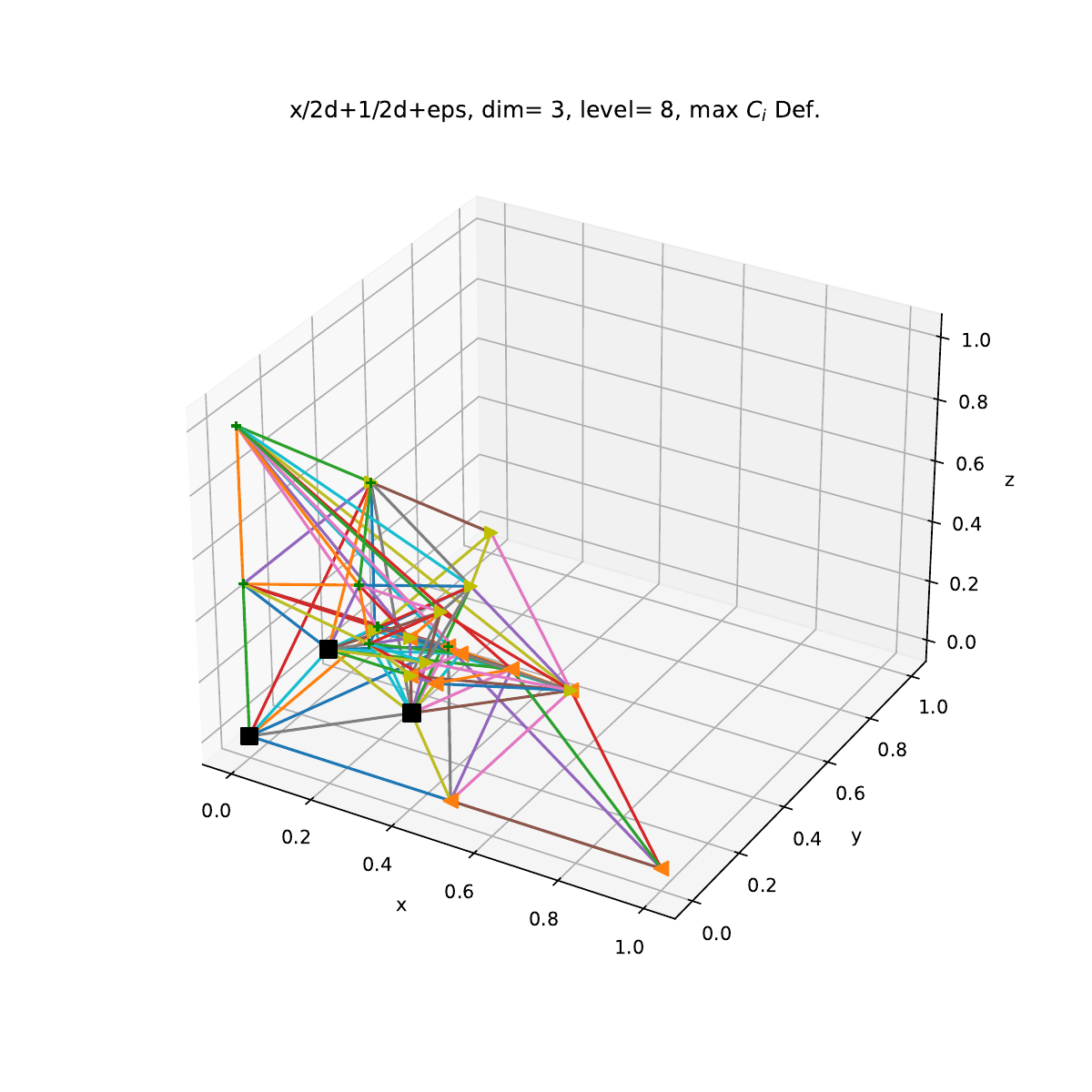}
\end{center}
\caption{\label{fig:x/4+1/4NewCi_3d}  Refinement pattern for example \eqref{eq:x/4+1/4+eps}: $(x,y,z)^\intercal  \mapsto  \frac{1}{6}\vek x+\frac{1}{6}\vek 1+\vek \epsilon$, using $C_i$ definition \eqref{eq:C_iAsMax}. Points concentrate early around the Fixed Point. See Fig. \ref{fig:x/6+1/6} for comparison.    }
\end{figure}

\subsection{Examples}
\subsubsection{Example with one fixed point}

Consider the problem in \eqref{eq:x/4+1/4} again. In Fig. \ref{fig:x/4+1/4NewCi}, the $C_i$ as defined by \eqref{eq:C_iAsMax} are shown: Consider the boundary between $C_2$ and $C_0$. There $\lambda_2(\vek x)-\lambda_2(F(\vek x))=\lambda_0(\vek x)-\lambda_0(F(\vek x))$ holds. A basic calculation (see Appendix Sec. \ref{sec:CiBound})
 yields $x_2=-\frac12x_1+\frac12$. The boundary between $C_1$ and $C_0$ is calculated analogously, and the one between $C_1$ and $C_2$ must be on the diagonal for symmetry reasons.  For comparison with the $C_i$ definition \eqref{eq:C_i}, see  figure \ref{fig:x/4+1/4}.\\
 
 \begin{table}[h!tb]
\begin{minipage}{0.525\textwidth}
\footnotesize
\begin{tabular}{|c c c c| c|}
\hline Nr& x & y & z & e \\ \hline
95&	0.18750&0.25000&0.25000 &0.0718\\
...&	0.25000&0.18750&0.12500 &0.0910\\
&	0.18750&0.25000&0.12500 &0.0910\\
&	0.21875&0.21875&0.18750 &0.0293\\
&	0.31250&0.15625&0.21875 &0.1221\\
&	0.15625&0.31250&0.21875 &0.1221\\
&	0.18750&0.21875&0.21875 &0.0293\\
&	0.28125&0.18750&0.15625 &0.0931\\
&	0.34375&0.21875&0.12500 &0.1632\\
&	0.18750&0.28125&0.15625 &0.0931\\
104&	0.21875&0.34375&0.12500 &0.1632 \\ 
\hline
\end{tabular}
\end{minipage}
\begin{minipage}{0.525\textwidth}
\footnotesize
\begin{tabular}{|c c c c| c|}
\hline Nr & x & y & z & e\\ \hline
16& 0.12500&0.12500&0.25000  & 0.1172	 \\
17 &0.25000&0.12500&0.37500  & 0.1968	 \\
18 &0.12500&0.25000&0.37500  & 0.1968	 \\
...&0.37500&0.25000&0.25000  & 0.1887	 \\
&0.25000&0.37500&0.25000     & 0.1887	 \\
&0.18750&0.18750&0.25000     & 0.0530	 \\
&0.25000&0.25000&0.12500     & 0.1030	 \\
&0.31250&0.25000&0.12500     & 0.1441	 \\
&0.25000&0.31250&0.12500     & 0.1441	 \\
25&0.21875&0.21875&0.12500   & 0.0795	 \\ 
26&0.21875&0.21875&0.18750   & 0.0293\\ \hline
\end{tabular}
\end{minipage}
\caption{\label{tab:ptsx2d+1_2d+eps} Points found  for problem  \eqref{eq:x/4+1/4+eps} 
in 3d ($\vek u_{FP} \sim \left(\frac15,\frac15,  \frac15\right)^\intercal$) Definition of $C_i$ according to \eqref{eq:C_i} (left) and \eqref{eq:C_iAsMax} (right), 9 steps of algorithm, $e:= \norm{\vek x-\vek x_{FP}}$. The last 11 points from 104 (left) and 26 (right) points produced are shown. The points found using definition of $C_i$ according to \eqref{eq:C_i} (left) are identical to those for the undisturbed Problem \eqref{eq:x/4+1/4}  in Tab. \ref{tab:ptsx2d+1_2d} due to overlapping $C_i$ but are presented here again for comparison.   }
\end{table}

\begin{figure}[h!tb]
\includegraphics[width=0.5\textwidth]{\pstricks/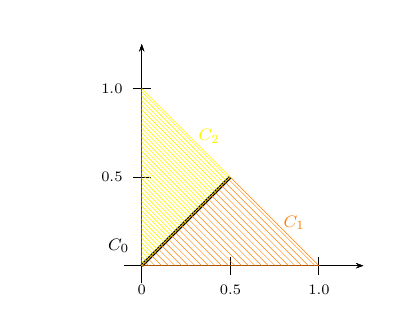}
\includegraphics[width=0.5\textwidth]{\maxCi/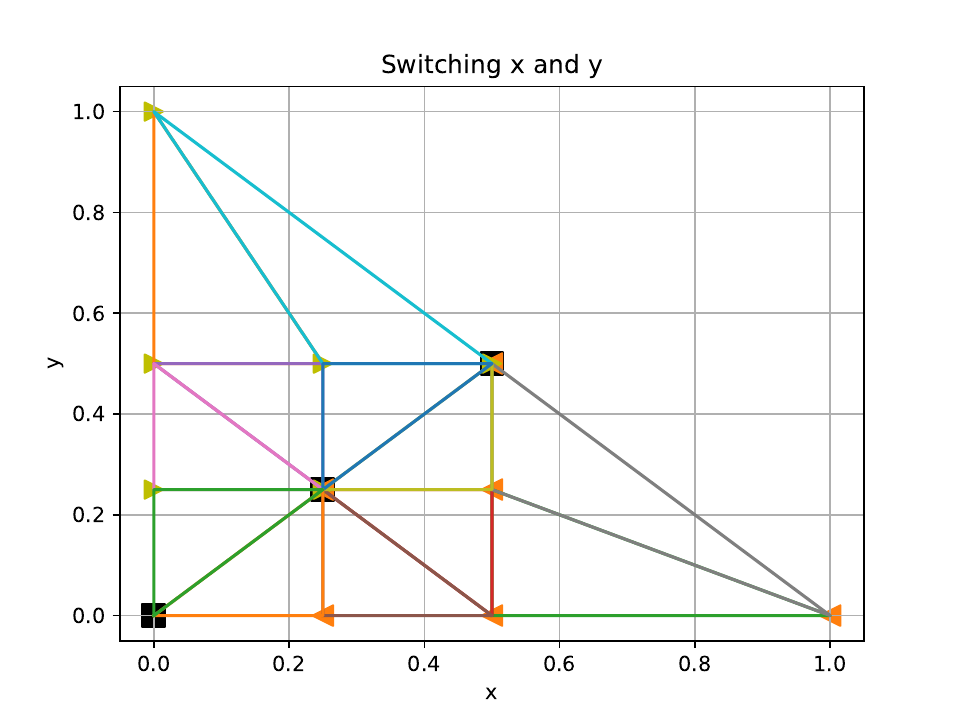}
\includegraphics[width=0.5\textwidth]{\maxCi/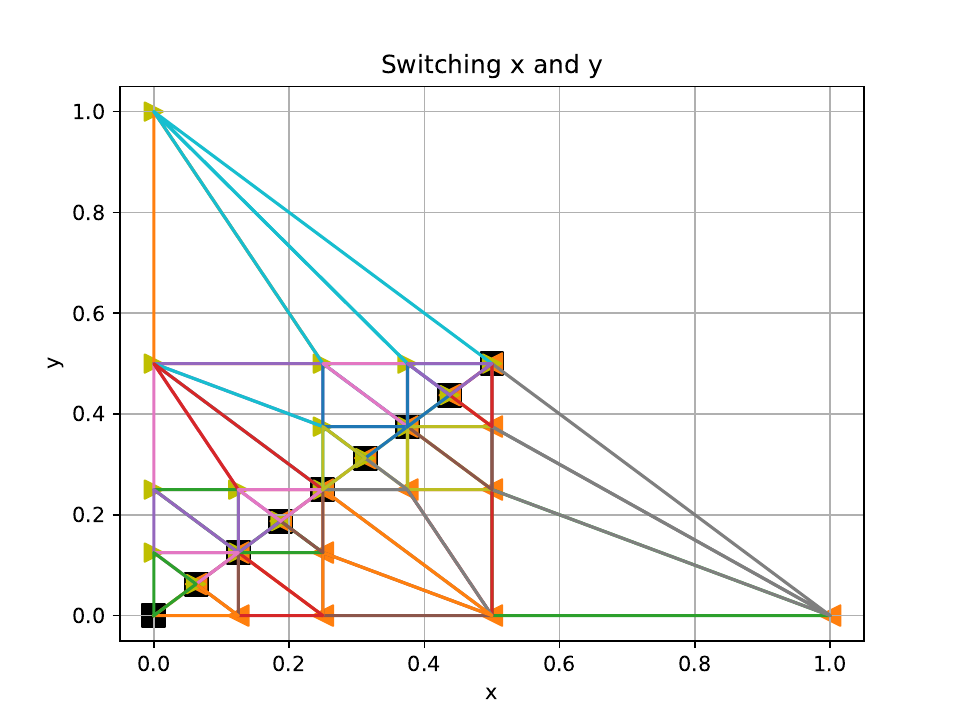}
\includegraphics[width=0.5\textwidth]{\maxCi/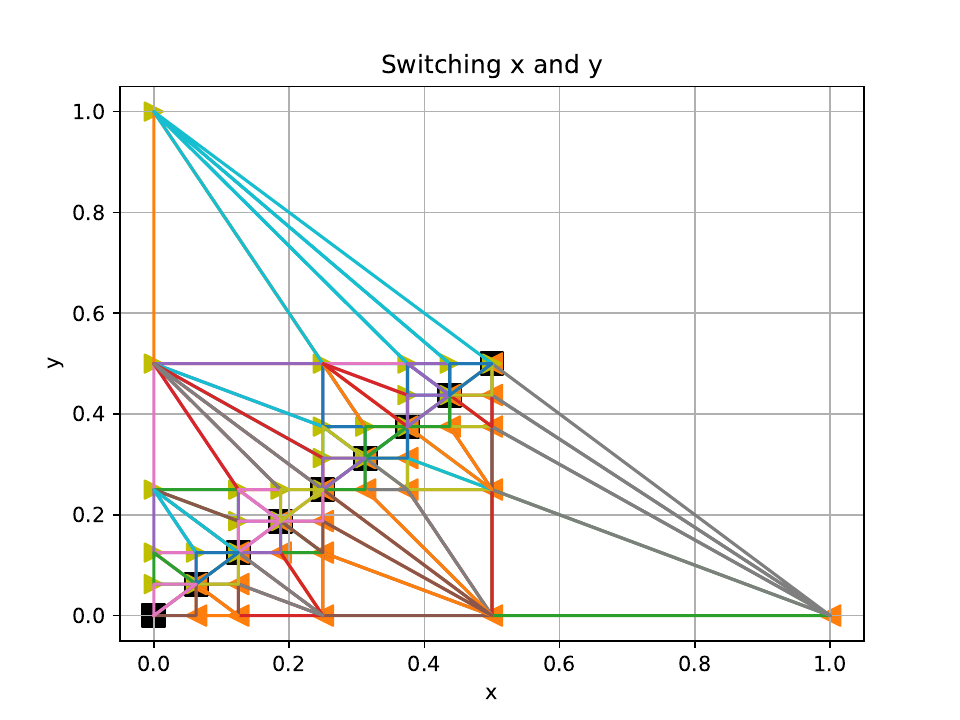}
\caption{\label{fig:switchXYMaxCi}Above left: The $C_i$ of example \eqref{eq:switchxy}: $(x,y)^\intercal  \mapsto  (y,x)^\intercal$ (switching $x$ and $y$) using $C_i$ definition \eqref{eq:C_iAsMax}.  $C_0$ is the diagonal only. $C_1$ and $C_2$ remain as in Fig. \ref{fig:switchXYMaxCi}.   Right and below: refinements using $C_i$ definition \eqref{eq:C_iAsMax}. See Fig. \ref{fig:switchxy} for comparison.}
\end{figure}

 Now problem \eqref{eq:x/4+1/4} is a simple test problem, and for its simplicity for example the diagonal is in $C_1$ and $C_2$, so all nodes on it are in more than one $C_i$ for both ways of defining $C_i$. Moreover,  $(0,0.5)^\intercal$ and $(0.5,0)^\intercal$ are in $C_0$ and $C_1$ respectively $C_2$ when using Def. \eqref{eq:C_iAsMax} as well. Thus typical points of a refinement are still in more than one $C_i$. This hides the advantage of Def. \eqref{eq:C_iAsMax}. We thus slightly modify the problem by disturbing it:
 \begin{equation}\label{eq:x/4+1/4+eps}
\vek y(\vek x)= \frac{\vek x}{2d}+\frac{1}{2d}\vek 1 +\vek \epsilon.
\end{equation}
The $C_i$ memberships do not change in a visible way, they look like in Fig. \ref{fig:x/4+1/4}. 
As the intersections of the $C_i$ according to the old defininition Eq. \eqref{eq:C_i} are quite wide and thus is essentially unaffected by the disturbance, the algorithm using the old $C_i$  yields an identical refinement as in Fig. \ref{fig:x/4+1/4}, so is neither presented again. Application of the new $C_i$ membership criterion \eqref{eq:C_iAsMax} yields a refinement pattern that is much more sparse (Figs. \label{fig:x/4+1/4NewCi_3d} \ref{fig:x/4+1/4NewCi}) and approaches the fixed point using less points (Tab. \ref{tab:ptsx2d+1_2d+eps}).

\subsubsection{The switching example}

As a further example,  the mapping that switches $x$ and $y$  $(x,y)^\intercal  \mapsto  (y,x)^\intercal$ \eqref{eq:switchxy} using $C_i$ definition \eqref{eq:C_iAsMax} is plotted in Fig. \ref{fig:switchXYMaxCi}. As on the diagonal $\lambda_0(\vek x)=\lambda_0(\vek f(\vek x))$ and thus $\lambda_0(\vek x)-\lambda_0(\vek f(\vek x))=0$, $C_0$ is the diagonal only. $C_1$ and $C_2$ remain as for the old $C_i$ definition. The refinement now is restricted to simplices that actually contain a fixed point.

The new refinement pattern is much more sparse and approaches the fixed point using less points.\\
However, there are situations in which a lot of Sperner Simplices will emerge. The characterization of such mappings should be considered in future work. 

\subsubsection{Fourdimensional example with one fixed point}
Example Eq. \eqref{eq:x/4+1/4+eps} is run for $d=4$. The results are presented in Table \ref{tab:ptsx2d+1_2d_4d}.
\begin{table}
\begin{minipage}{0.5\textwidth}

\begin{tabular}{|c c c c c| c |}
\hline Nr & $x_1$ & $x_2$ & $x_3$ & $x_4$  &$\norm{\vek x-\vek x_{FP}}$\\ \hline
80&	0.140625& 0.203125& 0.09375 & 0.15625 & 0.07891 \\
..&	0.125   & 0.171875& 0.109375& 0.125   & 0.05099 \\
&	0.125   & 0.15625 & 0.15625 & 0.15625 & 0.02927 \\
&	0.140625& 0.125   & 0.140625& 0.125   & 0.02545 \\
&	0.125   & 0.140625& 0.140625& 0.125   & 0.02545\\ 
&	0.15625 & 0.15625  &0.09375  &0.15625 & 0.05431 \\
&	0.140625& 0.125    &0.171875 &0.15625 & 0.03667 \\
&	0.125	& 0.140625 &0.171875 &0.15625 & 0.03667 \\
&	0.203125& 0.140625 &0.15625  &0.15625 & 0.06321 \\
&	0.140625& 0.203125 &0.15625  &0.15625 & 0.06321 \\
...&	0.140625& 0.15625  &0.140625 &0.15625 & 0.01920 \\
92&	0.15625 & 0.140625 &0.140625 &0.125   & 0.02254 \\
93&	0.140625& 0.15625  &0.140625 &0.125   & 0.02254 \\
94&	0.140625& 0.140625 &0.15625  &0.15625 & 0.01920 \\
95&	0.15625 & 0.171875 &0.109375 &0.125   & 0.04961 \\
 \hline
\end{tabular}
\end{minipage}
\caption{\label{tab:ptsx2d+1_2d_4d} Points found  for problem  \eqref{eq:x/4+1/4+eps}  
in 4d ($\vek u_{FP} = \left(\frac17,\frac17,  \frac17,   \frac17\right)^\intercal)$, $\frac17\sim0.142857$. Definition of $C_i$ according to \eqref{eq:C_i} (left) and \eqref{eq:C_iAsMax} (right), 9 steps of algorithm. The last 10 points from 95   points produced are shown. }
\end{table} 

Keeping in mind that it takes four steps to half all edge lengths of the simplex containing the fixed point and that there is probably more than one Sperner Simplex, the found points exhibit convergence.

\section{Brouwer's  Fixed Point Theorem on convex, compact sets: Mapping to Simplices }
The algorithm \ref{algo:knaster} applies to mappings $\mathcal S \longrightarrow \mathcal S$ only, whereas
Brouwer's Fixed Point Theorem applies to general compact hole-free sets $\Omega$, proven by arguing that there is always a bijective mapping $A: \Omega \longrightarrow \mathcal S $.
\begin{figure}[h!tb]
\begin{center}
\includegraphics[width=0.5\textwidth]{\pstricks/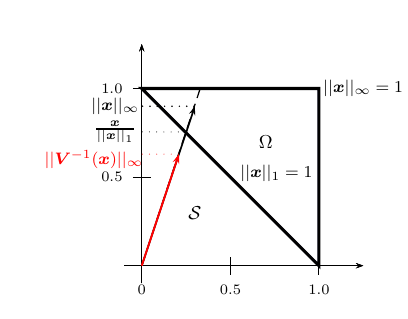}
\end{center}
\caption{\label{fig:fromOmToSimp} The Mapping $\vek V:\simp \longrightarrow \Omega$ }
\end{figure}

 Thus, as $\vek F:\simp \to \simp$ has a fixed point $\vek x$, the mapping
$$ A^{-1}\circ \vek F\circ A: \, \Omega \longrightarrow \Omega $$ 
has the fixed point  $\vek y= A^{-1}\vek x$.\\
The injectivity can be achieved for any zero search problem of a bounded mapping (Section \ref{sec:injectivity}).

For extending the algorithm to more general domains $\Omega$, this argumentation is the wrong way around: The algorithm relies on Th. \ref{thm:brouwerSimp} and consequently requires evaluations at corners of simplices, which one finds in the triangulation, while the function $\tilde F$ whose fixed points one wants to find is from $\Omega$ to $\Omega$.\\
Thus, a mapping $\vek V:\simp \longrightarrow \Omega$ is needed and the algorithm is applied on the mapping
\begin{align}
 \vek V^{-1}\circ\vek F \circ\vek V: \simp \longrightarrow & \simp\\
 \vek x \mapsto & \vek V^{-1}(\vek F(\vek V\vek x)).
\end{align}
The mapping $\vek V:\simp \longrightarrow \Omega$ is readily found for square domains, without loss of generality $\Omega = [0,1]^d$: Then, for $x_j$ being the maximum component of $\vek x$,
the relation 
$$\frac{x_j}{(\vek V^{-1}(\vek x))_j}=\frac{1}{\left(\frac{\vek x}{\norm{\vek x}_1}\right)_j}$$ 
holds (see Fig. \ref{fig:fromOmToSimp}),
yielding
\begin{equation}
(\vek V^{-1}(\vek x))_j =\left(\frac{x_j\vek x}{\norm{\vek x}_1}\right)_j
\end{equation}
or 
\begin{equation}
\vek V^{-1}(\vek x) =\frac{\norm{\vek x}_\infty}{\norm{\vek x}_1}\vek x.
\end{equation}

\section{Discussion}
The suggested  algorithm can not and is not meant to compete with Newton-type solvers  in settings where one looks for one specific zero which one has an idea of or even has a good starting solution.\\
The suggested algorithm does not necessarily find all solutions, but by applying additional refinements more  solutions can be found in a much more systematic and efficient way than by placing a grid of initial points for Newton-type solvers.

\subsection{Usage for optimization}
The algorithm has the property
 that the scaling by the maximum values of the function for achieving the injectivity of the mapping that is presented in Appendix \ref{sec:injectivity} can be repeated once new maximum components are found and the $C_i$ memberships can be recalculated, which is not computationally costly. In the worst case, simplices may have been misjudged as Sperner and been refined in vain.
It thus seems feasible to use the suggested algorithm for search of zeros of the gradient during optimisation. Difference quotients can be calculated from the objective function evaluated at nodes. Initially coarse approximations of the gradient, the difference quotients improve with the refinements, and  as suggested above, the $C_i$ memberships can be rejudged using a  posteriori rescaling at a low computational cost. A transformation to injectivity according to Sec. \ref{sec:injectivity} can be rescaled during the calculation as well.\\

\subsection{Explorative nature of the algorithm}

The strength of the algorithm is that it explores the function globally during the solution process, while Newton-type solvers deliver function evaluations only very close to the solution path, which means that their exploration is essentially one-dimensional. The number of function evaluations for the suggested algorithm is nominally higher than for Newton-type solvers, but the evaluated points are far more useful. The exploration of the suggested algorithm, though global, becomes finer near the points of interest. So the solving process produces not only a solution, but sensibly distributed function evaluations, useful for further usage in e.g. surrogate models such as neural networks. This can be useful in settings where evaluations are costly, e.g. if acquired through experiments. The algorithm then is a valuable Design-of-Experiment tool, as it suggests which evaluation point should be next for maximum information production.\\

\subsection{Possible enhancements}\label{sec:enhancements}
By choosing the bisection points on edges closer to $C_i$ boundaries, new points have closer distance to the fixed point than those placed  in the middle of edges. On the other hand, 
using for example \emph{regula falsi}, in case of convexity the new points might approach the fixed point from one side, so the content interval the fixed point is in might not go to zero and one would then interpret the results of the algorithm as a sequence of points rather than a sequence of simplices. Furthermore, this comes at the cost of exploration - points are much more concentrated in that case.

Another promising variation is not bisecting Sperner simplices, but placing a new point in the baricentrum of the simplex or the barycentrum of its image or a combination of both, so a simplex is divided into $d$ new simplices. Alternatively, the new point could be placed by solving the system $\vek\lambda(\vek x)- \vek\lambda(\tilde{\vek F}(\vek x))=\vek 0$, where $\tilde{\vek F}$ denotes the linear approximation of $\vek F$ through the corners of the Sperner simplex. \\
  A problem could be that older points of the refined simplex become corners of many simplexes, and if they are in more than one $C_i$, this would increase  the  number of produced Sperner Simplices. This will be investigated in future work. 

\subsection{Usage in combination with other solvers}
Once it becomes clear that a fixed point is close to a Sperner simplex, possibly by usage of smoothness information, the search can be continued by a Newton-Type solver. In fact, solving $\vek\lambda(\vek x)- \vek\lambda(\tilde{\vek F}(\vek x))=\vek 0$ for determining the next node as suggested in Sec. \ref{sec:enhancements} bears strong similarities with a Newton step as it solves a linear approximation of the problem. If this proves to work, a classical Newton solver seems to be obsolete.

\subsection{Exploitation of Mapping Degree Theory}

The suggested algorithm makes use of   the Mapping Degree Theory in the discrete form used to prove Fixed Point theorems as presented in Appendix \ref{sec:brouwerFP}.
\section{Implementation}
An implementation of  algorithm \ref{algo:knaster} is given in  the supplementary material or alternatively available from the author.
\section{Acknowledgements}
The author wants to thank Luisa Aust for her listening, reasoning and comments, but also for contributing imagery.

\appendix  
\appendixpage

\section{Transformations of Zero-Search Problems to injective Fixed-Point Form}
For the zero search problem 
\begin{align}\label{eq:zero}
\text{For }\vek G:\quad\realnr^d\supset \Omega &\longrightarrow  \realnr^d\\
\text{find } \vek x_0: \,\,\vek G(\vek x_0)& =\vek 0,
\end{align}
$\vek G $ bounded, the solution $\vek x_0$ is also the solution of 
\begin{align}\label{eq:injectivity}
\vek x +  \vek G(\vek x) & =\vek x
\end{align}
which has fixed point form, but in general is not injective.\\

\subsection{Transformation to injectivity} \label{sec:injectivity}
Now, a transformation to give \eqref{eq:zero} the injective mapping shape according to Eq. \eqref{eq:injection} is presented. We consider injectivity into $\simp $ only, so let $\vek x \in \simp.$\\
Let  $\vek U(\vek G) $ be a mapping that is positive, bounded and fulfils $\vek U (\vek 0) = \vek 0$, e.g. $\vek U(\vek G) = \vek G^2$ where the square denotes $\vek G^2_i=(\vek G)_i^2$. Then with
$$ c_i =\max_{\vek x \in \simp}\{ (\vek U)_{i}\}, $$ 
\begin{equation}
\begin{pmatrix}
\frac{1}{c_1} & & \\
& \ddots & \\
 & & \frac{1}{c_d}
\end{pmatrix}\vek U(\vek G)\le \vek 1 \quad \text{ and } \quad \vek U(\vek G)=\vek 0 \Leftrightarrow \vek G = \vek 0
\end{equation} 
where inequalities are to be read componentwise. Then 
\begin{equation}
\vek 0 \le - 0.9\begin{pmatrix}
\frac{1}{c_1}  U_1( G_1)& & \\
& \ddots & \\
 & & \frac{1}{c_d}U_d( G_d)
\end{pmatrix}
\vek x
+ \vek x \le \vek x , 
\end{equation} 
which is the injectivity of 
\begin{align}
\vek F :\simp &\longrightarrow \simp \\ 
\vek x &\mapsto \vek F :=- 0.9\begin{pmatrix}
\frac{1}{c_1}  U_1( G_1)& & \\
& \ddots & \\
 & & \frac{1}{c_d}U_d( G_d)
\end{pmatrix}
\vek x + \vek x,
\end{align}
and $\vek F=\vek 0$ if $\vek G=\vek 0$ or $\vek x=\vek 0$. The latter zero is unintended and needs separate treatment.

\section{Proof of Brouwer's Fixed Point theorem}\label{sec:brouwerFP}
The presented proof follows closely the proof given in \cite[Ch. 2]{zeidler95}, which again follows closely the 1929 original.
Consider a simplex  $\simp :=\text{conv}(\vek v_0,...\vek v_d)$ and a triangulation of it into subsimplices $\{\simp_j\}$. Let the corners $\vek u_i$ of all  $\{\simp_j\}$ be indexed in the following way: 
\begin{equation}\label{eq:indexedAsFace}
\vek u_i \in \text{conv}(\vek v_{i_1},...\vek v_{i_n})\,\Rightarrow \,i\in \{i_1,...,i_n\} 
\end{equation}
in words: a corner of $\{\simp_j\}$s number  is that of one of the corners of the edge, face, 3d-face and so on  of $\simp$ that contains that corner. For example, $i\in\{0,1\}$ iff $\vek u_i$ is on   the edge between $\vek v_0$ and $\vek v_1$.
\begin{definition}  \label{th:defSperner}
A simplex $\simp_j$ of the triangulation is called \emph{Sperner} iff its corners are indexed with 1,...,d.
\end{definition} 
\begin{theorem}[Sperner's Lemma] \label{th:sperner}Any triangulation with corners indexed as above contains an impair number of Sperner simplices.
\end{theorem}
\begin{proof} A $d-1$-dimensional subsimplex is called \emph{distinguished} iff it contains indices 1,2,...,d-1. Thus if any $d$-dimensional subsimplex contains
\begin{itemize}
\item exactly one distinguished $d-1$-subsimplex, then it is Sperner
\item or none or more than one, then it is not Sperner.
\end{itemize}
In 1d, the distinguished subsimplices are the 0-vertices.
One is on the boundary of $\simp$($\vek v_0$ itself), yielding (sooner or later) one Sperner simplex.
Others in the interior of  $\simp$ yield (sooner or later) two Sperner simplices. As there is only one distinguished subsimplex on the boundary, the statement follows.\\
In 2d, the distinguished subsimplices are the edges containing 0- and 1-knots. On  $\text{conv}(\vek v_0, \vek v_1)$, they are contained in one 2-simplex, so
\begin{itemize}
\item yielding one Sperner simplex each 
\item or a pair might produce none
\end{itemize}
so an impair number as the number of distinguished 1-simplices is impair.
In the interior, a distinguished simplex is contained in two simplices, so yielding two or no Sperner simplices.\\
For $d =3,4,...$, the  claim follows by induction by repeating the arguing that
\begin{itemize}
\item on  $\text{conv}(\vek v_0, ..., \vek v_{d-1})$ there is an impair number of $d-1$-Sperner simplices, so an impair number of distinguished subsimplices
\item those distinguished simplices on the $(\vek v_0, ..., \vek v_{d-1})$ boundary face yield an impair number of Sperner simplices
\item and those in the interior yield a pair number of Sperner simplices. 
\end{itemize}

\end{proof}

\begin{theorem}[The Lemma of Knaster, Kuratowski and Mazurkewicz] \label{th:kkm}
\end{theorem}
 Let   $\simp :=\text{conv}(\vek v_0,...\vek v_d)$ be a simplex and $\{C_i\}, \quad i=0,...d$ a set of sets with the property 
 \begin{equation}\label{eq:CiCover}
\text{conv}(\vek v_{i_1},...\vek v_{i_n}) \subset \bigcup_{i_1,...,i_n}C_{i_k}
\end{equation}
for any $\{i_1,...,i_n\} \subset \{0,1,...,d\}$.
Then there is a point $\vek u$ that is in all $C_i$ 
\begin{equation}
\vek u\in C_i, \quad i=0,...,d.
\end{equation}
\begin{proof}
Let   $\{\simp_j\}$ be a triangulation of $\simp$  and let each corner of it be indexed with the number of the $C_i$ it is in. Such an indexing exists as each corner is in at least one $C_i$ by \eqref{eq:CiCover}, and also by \eqref{eq:CiCover}, prerequisite \eqref{eq:indexedAsFace} is fulfilled. Thus, one $\simp_j$ of the triangulation is Sperner, meaning that it contains elements of all $C_i$.\\
By bisecting the triangulation, there is a sequence of Sperner simplices in which each element has  half the  diameter as its precessor. Thus the diameter of this sequence of Sperner simplices converges to zero, and there is a point $\vek u$ that is in all $C_i$.
\end{proof}

\begin{theorem}[Brouwer's Fixed Point theorem]\label{thm:brouwerSimp}
 Let 
 \begin{equation}\label{}
F:\quad \simp \longrightarrow \simp 
\end{equation}
be a continuous mapping. Then there is an $\vek x \in \simp$ with 
$$ F(\vek x) = \vek x .$$
\end{theorem}
\begin{proof}
Let $\lambda_i(\vek x)$ denote the $i$-th barycentric coordinate of $\vek x$, $i=0,...,d$. Define
\begin{equation}\label{eq:C_iAppendix}
 C_i:=\{\vek x: \lambda_i(F(\vek x )) \le \lambda_i(\vek x )\}, \quad i=0,...,d.
\end{equation}
In words, these are the sets of image points lying farther away from (precisely, not closer to) corner $ \vek v_i$ than the preimage.

It holds for the corners that $\vek v_i\in C_i$ and moreover \eqref{eq:CiCover},
 \begin{equation}
\text{conv}(\vek v_{i_1},...\vek v_{i_n}) \subset \bigcup_{i_1,...,i_n}C_{i_k}:
\end{equation}
Would that not be true, then for a $\vek x \in \text{conv}(\vek v_{i_1},...\vek v_{i_n})$ there would be no $\lambda_{i_k}$, $k=0,...n\le d$ such that 
\begin{align}
\lambda_i(F(\vek x )) \le& \lambda_i(\vek x )\\
\Leftrightarrow \quad \lambda_i(F(\vek x )) >& \lambda_i(\vek x ) \text{ for all }k=0,...n\le d
\end{align}
But this contradicts that for barycentric coordinates, $\sum_{i=0}^d\lambda_i=1$.\\
As those $C_i$ fulfil prerequisite equation  \eqref{eq:CiCover} of the lemma of Knaster, Kuratowski and Mazurkevic, it follows that there is a point $\vek x$ that is in all $C_i$. This means that for this $\vek x$,
 \begin{equation}
\lambda_i(F(\vek x ) \le \lambda_i(\vek x ) \text{ for all } i=0,...,d,
 \end{equation}
 and as  $\sum_{i=0}^d\lambda_i=1$,
 \begin{equation}\label{eq:C_iboundary}
\lambda_i(F(\vek x ) = \lambda_i(\vek x ) \text{ for all } i=0,...,d,
 \end{equation}
so $\vek x $ is a fixed point.

 \end{proof}
 
 \section{A brute force way to overcome Sperner simplex number explosion}\label{sec:reduceC_iBruteForce}
 To work around the redundant production of Sperner simplices around nodes that have multi-$C_i$ membership, it seems feasible to just discard all $C_i$ memberships of a node except one, e.g. the one with the smallest $i$. Repeating eq. \eqref{eq:C_iAppendix} for readability, 
\begin{equation*}
C_i:=\left\{\vek u \in S : \lambda_j(\vek F(\vek u))\le\lambda_j(\vek u)\right\},
\end{equation*}
the $C_i$ are the sets of points whose images are not closer to corner $\vek v_i$. For establishing the prerequisites of the KKM lemma \ref{th:kkm} and thus guarantee that at least one fixed point is found,  we have to show that 
\begin{enumerate}
\item the emerging indexing of nodes by the index of the remaining membership still has the property \eqref{eq:indexedAsFace} from the Sperner lemma, 
\begin{equation*}
\vek u_i \in \text{conv}(\vek v_{i_1},...\vek v_{i_n})\,\Rightarrow \,i\in \{i_1,...,i_n\} \end{equation*}
\item if $C_i$ memberships of a corner of multiple Sperner simplices  are discarded, at least one of those simplices remains Sperner. 
\end{enumerate}
The first  is given if not only \eqref{eq:CiCover} is fulfilled, i.e.
 \begin{equation*}
\text{conv}(\vek v_{i_1},...\vek v_{i_n}) \subset \bigcup_{i_1,...,i_n}C_{i_k},
\end{equation*}
meaning each face of $\simp$ is covered by its corners' $C_i$, but furthermore \emph{consist exclusively}
 of those $C_i$ (otherwise the wrong memberships may be discarded):
 \begin{equation*}
C_m\cap \text{conv}(\vek v_{i_1},...\vek v_{i_n}) =\emptyset \text{ iff }m\notin \{i_1,...,i_n\}.
\end{equation*}
In other words,   $C_m$ shall not intersect $\text{conv}(\vek v_{i_1},...\vek v_{i_n})$ unless $m\in\{i_1,...,i_n\}$. 
This is established in the following way: Let $\vek x \in \text{conv}(\vek v_{i_1},...\vek v_{i_n})$. Then for $m\notin\{i_1,...,i_n\}$ obviously $\lambda_m(\vek x)=0$. On the other hand, the minimum $\lambda_m(F(\vek x))=0$, which means that $F(\vek x)\in \text{conv}(\vek v_{i_1},...\vek v_{i_n})$ as $\vek x$. So  \eqref{eq:indexedAsFace} is given once all  $C_i$  memberships yielding from $\lambda_i(F(\vek x))=\lambda_m(\vek x)=0$ are discarded, and is not further affected if further memberships, e.g. all remaining memberships except one,  are discarded.
\begin{remark}
 $\lambda_m(\vek x)=0$ is a sufficient criterium for ignoring $C_m$ membership, because then either $\lambda_i(F(\vek x))=0$ or not, and in the latter case $\vek x $ is not in $C_m$  neither. \\
 Above argumentation assumed $\vek x \in \text{conv}(\vek v_{i_1},...\vek v_{i_n})$, $m\notin\{i_1,...,i_n\}$, which is is established by $\lambda_m(\vek x)=0$.
\end{remark}
Unfortunately, 
the second property is not given. Consider the $C_i$ plot Fig. \ref{fig:switchxy} of problem \eqref{eq:switchxy}. The interior points of the sets $C_2$ as well as $C_1$ are inside $C_0$ and so assigned to $C_0$. So using this reduction, there will be no Sperner simplices in the interior of $\simp$ and all fixed points along the diagonal will be lost.\\


\section{Calculation of the $C_2$-$C_0$ boundary for  $\vek F= \frac14\vek x + \frac 14 \vek 1 $}
\label{sec:CiBound}
The calculation of  the $C_2$-$C_0$ boundary for the problem in eq. \eqref{eq:x/4+1/4} is elementary but too long to be left to the reader. From
\begin{align*}\lambda_2(\vek x)-\lambda_2(F(\vek x))=&\lambda_0(\vek x)-\lambda_0(\vek F(\vek x))\\
=& 1-\lambda_2(\vek x) - \lambda_1(\vek x) - (1-\lambda_2(\vek F(\vek x)) -\lambda_1(\vek F(\vek x)))\\
\Leftrightarrow\quad 
2\lambda_2(\vek x)-2\lambda_2(\vek F(\vek x))=&-\lambda_1(\vek x)+\lambda_1(\vek F(\vek x)).
\end{align*}
Apply this to $$\vek F= \frac14\vek x + \frac 14 \vek 1 $$
using $\lambda_1=x_1$ and $\lambda_2=x_2$:
\begin{align*}
 2x_2 -2\left(\frac14x_2+\frac14\right)=&- x_1 -2\left(\frac14x_1+\frac14\right)\\
\Leftrightarrow\quad 4x_2 -2 =& -3x_1+1\\
\Leftrightarrow\quad x_2  =& -\frac12 x_1+\frac12.\\
\end{align*}

\bibliographystyle{plain}
\bibliography{./lit.bib}

\begin{thebibliography}{1}

\bibitem{zeidler95}
Eberhard~H. Zeidler.
\newblock {\em Applied Functional Analysis: Applications to Mathematical
  Physics}.
\newblock Springer New York, NY, 1995.

\end{thebibliography}
\end{document}